\documentclass{amsart}
\usepackage{amsmath}
\usepackage{amssymb}
\usepackage{amsthm}
\usepackage{epsfig}
\usepackage{amsfonts}
\usepackage{amscd}
\usepackage{mathrsfs}
\usepackage{enumerate}
\usepackage{latexsym}
\usepackage{url}

\newtheorem{theorem}{Theorem}[section]
\newtheorem{lemma}[theorem]{Lemma}
\newtheorem{proposition}[theorem]{Proposition}
\newtheorem{corollary}[theorem]{Corollary}
\newtheorem{question}[theorem]{Question}

\theoremstyle{definition}
\newtheorem{definition}[theorem]{Definition}

\theoremstyle{remark}
\newtheorem{remark}[theorem]{Remark}
\numberwithin{equation}{section}

\begin{document}

\title[One-point extensions of a locally compact space]{On order structure of the set of one-point Tychonoff
extensions of a locally compact space}

\author{M.R. Koushesh}
\address{Department of Mathematical Sciences, Isfahan University of Technology, Isfahan 84156--83111, Iran}
\email{koushesh@cc.iut.ac.ir}

\subjclass[2000]{54D35, 54D40, 54A25}

\keywords{One-point extension; One-point compactification; Stone-\v{C}ech compactification; $\beta X\backslash X$; Rings of continuous functions.}

\begin{abstract}
If a Tychonoff  space $X$ is dense in a Tychonoff space $Y$, then $Y$ is called  a {\em Tychonoff extension} of $X$. Two Tychonoff extensions $Y_1$ and $Y_2$ of $X$ are said to be {\em equivalent}, if there exists a homeomorphism $f:Y_1\rightarrow Y_2$ which keeps $X$ pointwise fixed. This defines an equivalence relation on the class of all  Tychonoff extensions of $X$. We identify those extensions of $X$ which belong to the same equivalence classes. For two Tychonoff extensions $Y_1$ and $Y_2$ of $X$, we write $Y_2\leq Y_1$, if there exists a continuous function $f:Y_1\rightarrow Y_2$ which keeps $X$ pointwise fixed. This is a  partial order on the set of all (equivalence classes of) Tychonoff extensions of $X$. If a Tychonoff extension $Y$ of $X$ is such that $Y\backslash X$ is a singleton, then $Y$ is called a {\em one-point extension} of $X$. Let ${\mathcal T}(X)$ denote the set of all one-point extensions of
$X$. Our purpose is to study the order structure of  the partially
ordered set $({\mathcal T}(X), \leq)$. For a locally compact space
$X$, we define an order-anti-isomorphism from ${\mathcal T}(X)$ onto
the set of all non-empty closed subsets of $\beta X\backslash X$. We
consider various sets of one-point extensions, including the set of
all
        one-point locally compact  extensions of $X$, the set of all
        one-point  Lindel\"{o}f extensions of $X$, the  set of all
        one-point  pseudocompact extensions of $X$, and the set of all
        one-point  \v{C}ech-complete extensions of $X$, among
        others. We study how these sets of one-point extensions are
related, and
        investigate the relation between their
order structure,  and the topology of subspaces of $\beta
X\backslash X$. We find some lower bounds for cardinalities of some
of these sets of one-point extensions, and in a concluding section,
we show how some of our results may be applied to obtain relations
between the order structure of certain subfamilies of ideals of
$C^*(X)$, partially ordered with inclusion, and the topology of
subspaces of $\beta X\backslash X$. We leave some problems open.
\end{abstract}

\maketitle

\section{Introduction}

Let $X$ be a Tychonoff  space. If a Tychonoff space $Y$ contains $X$
as a dense subspace, we call $Y$ a {\em Tychonoff extension} of $X$.
Two Tychonoff extensions $Y_1$ and $Y_2$ of $X$ are said to be {\em
equivalent}, if there exists a homeomorphism $f:Y_1\rightarrow Y_2$
which keeps $X$ pointwise fixed. This indeed defines an equivalence
relation which splits the class of all Tychonoff extensions of $X$
into equivalence classes. We identify the equivalence classes with
individuals whenever no confusion arises. For two Tychonoff
extensions $Y_1$ and $Y_2$ of $X$, we write $Y_2\leq Y_1$, if there
exists a continuous function $f:Y_1\rightarrow Y_2$ which keeps $X$
pointwise fixed. This defines a partial order on the set of all
(equivalence classes of) Tychonoff extensions of $X$. A detailed
study of this partial order can be found in Section \ref{YROEJ} of [12]. If
an extension $Y$ of $X$ is such that $Y\backslash X$ consists of a
single element, then $Y$ is called a {\em one-point extension} of
$X$. Let ${\mathcal T}(X)$ denote the set of all one-point
extensions of $X$. Our purpose here is to study the order structure
of  the partially ordered set $({\mathcal T}(X), \leq)$.

First we define some notations and terminologies we will use. For a
Tychonoff space $X$, we let $\beta X$ and $\upsilon X$ denote the
Stone-\v{C}ech compactification of $X$ and the Hewitt
realcompactification of $X$, respectively. For a subset $A$ of $X$,
we let $A^*=(\mbox {cl}_{\beta X}A)\backslash X$. In particular,
$X^*=\beta X\backslash X$. For a space $X$, we denote the set of all
closed subset of $X$, the set of all zero-sets of $X$, and the set
of all clopen (open-closed) subsets of $X$, by ${\mathcal C}(X)$,
${\mathcal Z}(X)$ and ${\mathcal B}(X)$, respectively. A space $X$
is called {\em zero-dimensional}, if the set  ${\mathcal B}(X)$ is
an open base for $X$. For two spaces $X$ and $Y$, $C(X, Y)$ denotes
the set of all continuous functions from $X$ to $Y$. The letters
$\mathbf{I}$ and $\mathbf{R}$, denote the closed unit interval and
the real line, respectively. We also let $C(X)=C(X, \mathbf{R})$,
and we denote by $C^*(X)$ the set of all bounded elements of $C(X)$.

We denote by $\omega$ the first countably infinite ordinal number,
and we denote by $\aleph_0$ the cardinality of $\omega$. By [CH] and
[MA] we mean the Continuum Hypothesis and the Martin's Axiom, and
whenever they appear at  the beginning of the statement of a
theorem, indicate that they have been assumed in the proof of that
theorem.

In a partially ordered set $P$, the two symbols $\bigvee$ and
$\bigwedge$ are used to denote the least upper bound and the
greatest lower bound (provided that they exist) respectively. The
elements $\bigvee P$ and $\bigwedge P$  are called the {\em maximum}
and the {\em minimum} of $P$, respectively. An element $p\in P$ is
called a {\em maximal} ({\em minimal}, respectively) element of $P$,
provided that for every $x\in P$, if $x\geq p$ ($x\leq p$,
respectively) then $x=p$. If $P$ and $Q$ are partially ordered sets,
a function $f:P\rightarrow Q$ is called an {\em order-homomorphism}
({\em order-anti-homomorphism}, respectively) if $f(a)\leq f(b)$
($f(a)\geq f(b)$, respectively) whenever $a\leq  b$. The function
$f$ is called  an {\em order-isomorphism} ({\em
order-anti-isomorphism}, respectively) if it is moreover bijective,
and $f^{-1}:Q\rightarrow P$ also is an order-homomorphism
(order-anti-homomorphism, respectively). The partially ordered sets
$P$ and $Q$ are called {\em order-isomorphic} ({\em
order-anti-isomorphic}, respectively) if there is an
order-isomorphism (order-anti-isomorphism, respectively) between
them.

For terms and notations not defined here we follow the standard text
of [4]. In particular, compact and paracompact spaces are assumed to
be Hausdorff, and perfect mappings are assumed to be continuous. By
a neighborhood of a point $x$ in a space $X$, we mean a subset of
$X$ which contains an open subset containing $x$.

In 1924, P. Alexandroff proved that a locally compact non-compact
space $X$ has a one-point extension which is compact. This is now
known as the Alexandroff compactification of $X$, or the one-point
compactification of $X$. Since then, one-point extensions have been
studied extensively by various authors (for some results as well as
some bibliographies on the subject see [10] and [11]). The majority
of these works,
however, deals with conditions under which if a space locally
possesses a topological
  property $\mathcal{P}$, then it has a one-point extension which has
  $\mathcal{P}$. Recently, M. Henriksen, L. Janos and R.G. Woods have
studied the
  partially ordered set of all one-point metrizable extensions of a
  locally compact metrizable space, by relating it to the topology of
subspaces of
  $X^*$. Here is a brief summary of the method they applied. Let $X$
  be a (non-compact)  metrizable space. We call a sequence
  $\{U_n\}_{n<\omega}$ of non-empty open subsets of $X$ an {\em
extension
  trace} in $X$, if for each $n<\omega$ we have
  $\mbox{cl}_XU_{n+1}\subseteq U_n$ and
$\bigcap_{n<\omega}U_n=\emptyset$.
  To every one-point metrizable extension $Y=X\cup\{p\}$ of $X$, we can
  correspond an extension trace of $X$, namely, $U_n=B(p,1/n)\cap
  X$. Conversely, if $\{U_n\}_{n<\omega}$ is an extension trace in
$X$,
  let $Y=X\cup\{p\}$, where $p\notin X$, and define a topology on $Y$
  consisting of sets of the form $V\cup\{p\}$, where $V$ is open in
$X$ and is such that
    $V\supseteq U_n$, for some $n<\omega$. By Theorem 2 of [1] (or
Theorem 3.4 of [2]) the space
    $Y$ thus defined is metrizable, and therefore it is a one-point
metrizable
      extension of $X$.
    It may happen, however, that
    different extension traces in $X$ give rise to the same
    one-point extension of $X$. To fix this problem, we define
    an equivalence relation on the set of all extension traces of
    $X$. Two extension traces $\{U_n\}_{n<\omega}$ and
    $\{V_n\}_{n<\omega}$ of $X$ are said to be {\em equivalent}, if
for
    each $n<\omega$, there exist $k_n,l_n<\omega$, such that
    $U_n\supseteq V_{k_n}$ and $V_n\supseteq U_{l_n}$. This makes a
(one-one) correspondence between the set of all equivalence
    classes of extension traces of $X$, and the set of all one-point
    metrizable extensions ${\mathcal E}(X)$ of $X$.  Using this, we
can define
      a function $\lambda: {\mathcal E}(X)\rightarrow{\mathcal Z}(X^*)$
      by $\lambda(Y)=\bigcap_{n<\omega}U_n^*$, where
$\{U_n\}_{n<\omega}$
      is an extension trace in $X$ which generates $Y$. It is proved in
[7]
      that the function
      $\lambda$ is well-defined, and it is an order-anti-isomorphism
      onto its image (in the case when $X$ is moreover separable, it is
proved in [7] that $\lambda$ maps ${\mathcal E}(X)$
      onto ${\mathcal Z}(X^*)\backslash\{\emptyset\}$). Using the
      function $\lambda$, and the fact that
      the topology of any compact space determines and is
      determined by the order structure of the set of its all
zero-sets,
      the authors of [7] have studied  the order
      structure of sets of one-point metrizable
      extensions of a locally compact metrizable space $X$, by
      relating it to the topology of certain subspaces of $X^*$.
      Motivated by the results of [7] and the author's earlier work
[8]
      (which is in fact a continuation of the work
        of Henriksen, Janos and Woods, in which we
        generalized most of the results of [7] from
        the separable case to the non-separable case)
        we  define an order-anti-isomorphism
        $\mu: {\mathcal T}(X)\rightarrow{\mathcal
        C}(X^*)\backslash\{\emptyset\}$. Using the mapping $\mu$ we
will be able
        to relate the topology of certain subspaces of $X^*$ to the
        order structure of various sets of one-point extensions of
$X$.
        These sets of one-point extensions include, the set of all
        one-point locally compact  extensions of $X$, the set of all
        one-point  Lindel\"{o}f extensions of $X$, the  set of all
        one-point  pseudocompact extensions of $X$, and the set of all
        one-point  \v{C}ech-complete extensions of $X$, among
        others.

In Section 2, we define certain sets of one-point extensions. We
then establish the order-isomorphism $\mu$, mentioned above,  from
the set of all one-point extensions of the  locally compact space
$X$ onto the set of all non-empty closed subsets of $X^*$. We find
the image under $\mu$ of some of the sets of one-point extensions
introduced before, and as a result, we show that the order structure
of some of them determines and is determined by the topology of the
space $X^*$. In Section 3, we obtain some results relating the order
structure of certain sets of one-point extensions of $X$ and the
topology of subspaces of $X^*$, under the extra assumption of
paracompactness of $X$. Section 4 deals with the order theoretic
relations between various sets of one-point extensions the space
$X$. In Section 5, we find sufficient conditions that some of the
sets of one-point extensions admit maximal or minimal elements. In
section 6, we find a lower bound for cardinalities of two of the
sets of one-point extensions introduced before. And finally, in
Section 7, we define an order-isomorphism from the set of all
Tychonoff extensions of a Tychonoff space  $X$ into the set of all
ideals of $C^*(X)$, partially ordered with inclusion. Using this, we
show how some of our previous results may be applied to obtain
relations between the order structure of certain subfamilies of
ideals of $C^*(X)$, partially ordered with inclusion, and the
topology of subspaces of $X^*$.

\section{The partially ordered set of one-point Tychonoff extensions of
a locally compact space}

The purpose of this article is to study the order structure of
various sets of one-point extensions. To make reference to these
sets easier, we list them all in the following definition.

\begin{definition}\label{FGA}
For a Tychonoff space $X$, let $ {\mathcal T}(X)$ denote the set of all one-point Tychonoff extensions of $X$. We define
\begin{itemize}
  \item ${\mathcal T}_C(X)=\{Y\in {\mathcal T}(X): Y \mbox{ is \v{C}ech-complete}\}$
  \item ${\mathcal T}_K(X)=\{Y\in {\mathcal T}(X): Y \mbox{ is locally compact}\}$
  \item ${\mathcal T}_D(X)=\{Y=X\cup\{p\}\in {\mathcal T}(X):\mbox {$p\notin\mbox{cl}_Y A$ for any  closed Lindel\"{o}f $A\subseteq X$}\}$
  \item ${\mathcal T}_L(X)=\{Y\in {\mathcal T}(X): Y \mbox{ is Lindel\"{o}f}\}$
  \item ${\mathcal T}_S(X)=\{Y=X\cup\{p\}\in {\mathcal T}(X):\\U\backslash\{p\}\mbox { is  $\sigma$-compact for some neighborhood $U$ of $p$ in $Y$}\}$
  \item ${\mathcal T}_P(X)=\{Y\in {\mathcal T}(X): Y \mbox{ is pseudocompact}\}$
  \item $ {\mathcal T}^*(X)=\{Y=X\cup\{p\}\in {\mathcal T}(X): Y \mbox{ is first countable at }p\}$
  \item ${\mathcal T}^*_{C}(X)={\mathcal T}^*(X)\cap{\mathcal T}_C(X)$, ${\mathcal T}^*_{K}(X)={\mathcal T}^*(X)\cap{\mathcal T}_K(X)$, ${\mathcal T}^*_{S}(X)={\mathcal T}^*(X)\cap{\mathcal T}_S(X)$, ${\mathcal T}_{CL}(X)={\mathcal T}_C(X)\cap{\mathcal T}_L(X)$ and ${\mathcal T}_{KL}(X)={\mathcal T}_K(X)\cap{\mathcal T}_L(X)$.
\end{itemize}
\end{definition}

For a space $X$, let ${\mathcal C}(X)$ denote the set of all closed subsets of $X$. Suppose that $X$ is a locally compact space. For each
$Y=X\cup\{p\}\in{\mathcal T}(X)$ let $F_Y:\beta X\rightarrow \beta
Y$ be the unique continuous function such that $F_Y|X=\mbox {id}_X$.
Define a function
\[\mu:\big({\mathcal T}(X), \leq\big)\rightarrow \big({\mathcal C}(X^*)\backslash\{\emptyset\},\subseteq\big)\]
by $\mu(Y)=F_Y^{-1}(p)$. In the following theorem we show
that the function $\mu$ so defined is an order-anti-isomorphism.
This is in fact a special case of Theorem \ref{YPPJ} of [10]. We give a
direct proof in here for the sake of completeness.

\begin{theorem}\label{YUS}
The function $\mu$ is an order-anti-isomorphism.
\end{theorem}

\begin{proof}
First we show that $\mu$ is onto. So  suppose that
$C\in{\mathcal C}(X^*)\backslash\{\emptyset\}$ and let $Z$ be the
space obtained from $\beta X$ by contracting $C$ to a point $p$. Let
$q:\beta X\rightarrow Z$ be the natural quotient mapping. Consider
$Y=X\cup\{p\}\subseteq Z$. Then since $Z$ is Tychonoff,
$Y\in{\mathcal T}(X)$.  We verify that $\mu(Y)=C$. First we note
that $Z=\beta Y$. This is because $Z$ is a compactification of $Y$
and every continuous function from $Y$ to ${\bf I}$  is continuously
extendable over $Z$. For if $h\in C(Y, {\bf I})$, let $g=hq:X\cup
C\rightarrow {\bf I}$  and let $G\in C(\beta X, {\bf I})$ be the
extension of $g$.  Let the function $H:Z\rightarrow {\bf I}$ be
defined by $H|(\beta X\backslash C)=G$ and $H(p)=h(p)$. Then $H$ is
a continuous extension of $h$. Now since $q|X=\mbox {id}_X=F_Y|X$,
we have $q=F_Y$, and therefore $\mu(Y)=F_Y^{-1}(p)=q^{-1}(p)=C$.

Now suppose that $Y_i=X\cup\{p_i\}\in{\mathcal T}(X)$, for $i=1,2$.
Suppose that $Y_1\geq
  Y_2$, and let $k:Y_1\rightarrow Y_2$ be a continuous function which
leaves $X$ pointwise
  fixed. Let $K :\beta Y_1\rightarrow \beta Y_2$ be the continuous
  extension of $k$. We denote $F_i=F_{Y_i}$, and we let $L=KF_1 :\beta
X\rightarrow \beta
  Y_2$. Then since $L|X=F_2|X$, we have $L=F_2$ and therefore since
  $K(p_1)=k(p_1)=p_2$ we obtain
\[ \mu (Y_2)= F_2^{-1}(p_2)= L^{-1}(p_2)=F_1^{-1}K^{-1}(p_2)\supseteq F_1^{-1}(p_1)=\mu (Y_1).\]

Next suppose that  $\mu (Y_2)\supseteq \mu (Y_1)$ and let
$ C_i=\mu (Y_i)$, for $i=1,2$. Let $Z_i$ be the quotient space
obtained from $\beta X$ by identifying each fiber  of $F_i=F_{Y_i}$
to a point and let $q_i:\beta X\rightarrow Z_i$ denote its
corresponding natural quotient mapping. Let the function
$G_i:Z_i\rightarrow \beta Y_i$ be defined by $G_i(F_i^{-1}(y))=y$.
Then $G_i$ is a continuous bijection, and since $Z_i$ is compact, it
is a homeomorphism which since $F_i(X^*)=\beta Y_i\backslash X$ (see
Theorem 3.\ref{PPG} of [4]) keeps $X$ pointwise fixed and $G_i(C_i)=p_i$.
We identify $Y_i$ with a subspace  of $Z_i$ under this
homeomorphism. Let $f:Y_1\rightarrow Y_2$ be a function such that
$f|X=\mbox {id}_X $ and $f(p_1)=p_2$. We verify that $f$ is
continuous. So suppose that $V$ is an open neighborhood of $p_2$ in
$Z_2$. Let $U=(Z_1\backslash q_1(\beta X\backslash q_2^{-1}(V)))\cap
Y_1$, which is an open subset of $Y_1$. Since
$q_2(C_2)=\{p_2\}\subseteq V$ and $C_1\subseteq C_2$, we have
$q_1^{-1}(p_1)=C_1\subseteq q_2^{-1}(V)$, and thus $p_1\notin
q_1(\beta X\backslash q_2^{-1}(V))$. Therefore $U$ is an open
neighborhood of $p_1$ in $Y_1$. Now since $U\cap X\subseteq V\cap
X$, we have $f(U)\subseteq V$, and thus $f$ is continuous at $p_1$.
Clearly $f$ is continuous on $X$ and therefore $Y_1\geq Y_2$. This
now completes the proof.
\end{proof}

Let $X$ be a locally compact space and let $Y\in {\mathcal T}(X)$
and $C=\mu (Y)$. If $Z$ is the space which is obtained from $\beta
X$ by contracting $C$ to a point $p$ and $q:\beta X\rightarrow Z$ is
its natural quotient mapping, then as the proof of Theorem \ref{YUS}
shows, we have $Y=X\cup\{p\}\subseteq Z$, $Z=\beta Y$ and $q=F_Y$.

\begin{remark}\label{GGJ}
If $X$ is a locally compact metrizable space then,
using the notations introduced in the introduction, we have
${\mathcal T}^*(X)={\mathcal E}(X)$ and $\mu|{\mathcal
T}^*(X)=\lambda$. The first assertion follows from Theorem 2 of [1].
We verify that $\mu|{\mathcal T}^*(X)=\lambda$.

Suppose that $Y=X\cup\left\{p\right\}\in{\mathcal E}(X)$ and let
$\left\{U_n\right\}_{n<\omega}$ be an extension trace in $X$ which
generates $Y$. Let $C=\mu(Y)$. We show that
$C=\bigcap_{n<\omega}U_n^*=\lambda(Y)$. First we verify that
$C\subseteq\bigcap_{n<\omega}U_n^*$. Suppose to the contrary that
there exists an $x\in C$ such that $x\notin \mbox{cl}_{\beta X}U_n$,
for some $n<\omega$. Now $U_n\cup\left\{p\right\}$ is an open
neighborhood of $p$ in $Y$. Let $V$ be an open subset of $\beta Y$
such that $U_n\cup\left\{p\right\}=V\cap Y$, and let $U$ be an open
neighborhood of $x$ in $\beta X$ such that $F_Y(U)\subseteq V$. Now
since $U\backslash  \mbox{cl}_{\beta X}U_n$ contains $x$, it is
non-empty. Let $t\in (U\backslash \mbox{cl}_{\beta X}U_n)\cap X$.
Then $t=F_Y(t)\in V$ and thus $t\in U_n$. But this contradicts the
choice of $t$. This shows that $x\in \bigcap_{n<\omega}U_n^*$, and
therefore $C\subseteq\bigcap_{n<\omega}U_n^*$.

To show the reverse inclusion, let $x\in \bigcap_{n<\omega}U_n^*$.
Suppose that $x\notin C$. Let $U$ and $V$ be open subsets of $\beta
X$ such that $x\in U$, $C\subseteq V$ and $\mbox{cl}_{\beta
X}U\cap\mbox{cl}_{\beta X}V=\emptyset$. Now since $(V\backslash
C)\cup\left\{p\right\}$ is an open neighborhood of $p$ in $\beta Y$,
there exists a $k<\omega $ such that
$U_k\cup\left\{p\right\}\subseteq (V\backslash
C)\cup\left\{p\right\}$. Therefore  $\mbox{cl}_{X}U_k\subseteq
V\backslash C$ and thus $\mbox{cl}_{X}U_k\cap \mbox{cl}_{X}(U\cap
X)=\emptyset$, as
\[\mbox{cl}_{X}U_k\cap \mbox{cl}_{X}(U\cap X)\subseteq
\mbox{cl}_{\beta X}U_k\cap \mbox{cl}_{\beta X}U.\]
But since $X$ is metrizable, this implies that
\[\mbox{cl}_{\beta X}U_k\cap \mbox{cl}_{\beta X}U= \mbox{cl}_{\beta
X}U_k\cap \mbox{cl}_{\beta X}(U\cap X)=\emptyset\]
which is a contradiction,  as  $x\in \mbox{cl}_{\beta
X}U_k\cap \mbox{cl}_{\beta X}U$. Therefore $x\in C$ and thus
$\bigcap_{n<\omega}U_n^*\subseteq C$. This together with the first
part shows that equality holds in the latter, which proves our
assertion.
\end{remark}

\begin{theorem}\label{DDS}
Let $X$ be a locally compact space. Then
\[\mu\big({\mathcal T}_C(X)\big)={\mathcal Z}(X^*)\backslash\{\emptyset\}.\]
\end{theorem}

\begin{proof}
Suppose that $Y \in{\mathcal T}_C(X)$. Then since $Y$
is \v{C}ech-complete, $Y$ is a $G_\delta$-set in $\beta Y$.
Therefore $X\cup \mu (Y)=F_Y^{-1}(Y)$ is a $G_\delta$-set in $\beta
X$ and thus $\mu (Y)$ is a closed $G_\delta$-set in $X^*$. Therefore
$\mu (Y)$ is a zero-set in $X^*$.

For
the reverse inclusion, suppose that $D\in {\mathcal
Z}(X^*)\backslash \{\emptyset\}$.  By the previous theorem
$D=\mu(Y)$, for some $Y \in{\mathcal T}(X)$. Let $X^*\backslash
D=\bigcup_{n<\omega} K_n$, where each $K_n$ is compact. Then we have
\[Y=F_Y\Big(\beta X\backslash\bigcup_{n<\omega}
K_n\Big)=\bigcap_{n<\omega}\big(\beta Y\backslash F_Y(K_n)\big)\]
which is a $G_\delta$-set in $\beta Y$. Thus $Y$ is
\v{C}ech-complete.
\end{proof}

Since for a compact space $X$, the order structure of either of the
sets ${\mathcal C}(X)$ or ${\mathcal Z}(X)$ determine the topology
of $X$, from Theorems \ref{YUS} and 2.3 we obtain the following result.

\begin{theorem}\label{TDJ}
For locally compact spaces $X$ and $Y$ the following conditions are equivalent.
\begin{itemize}
\item[\rm(1)] ${\mathcal T}(X)$ and ${\mathcal T}(Y)$ are order-isomorphic;
\item[\rm(2)] ${\mathcal T}_C(X)$  and  ${\mathcal T}_C(Y)$ are order-isomorphic;
\item[\rm(3)] $X^*$ and $Y^*$  are homeomorphic.
\end{itemize}
\end{theorem}

\begin{theorem}\label{TODJ}
Let $X$ be a locally compact
space. Then
\[ \mu\big({\mathcal T}_K(X)\big)={\mathcal B}(X^*)\backslash\{\emptyset\}.\]
\end{theorem}

\begin{proof}
Suppose that $Y=X\cup\{p\} \in{\mathcal T}_K(X)$. Then
since $Y$ is locally compact, $Y$ is open in $\beta Y$, and thus
$F^{-1}_Y(Y)$ is open in  $\beta X$. But since $F_Y( X^*)=\beta
Y\backslash X$, we have  $F_Y^{-1}(Y)=X\cup F_Y^{-1}(p)$, and
therefore $\mu (Y)= F^{-1}_Y(Y)\backslash X$ is open in $X^*$. Thus
$\mu (Y)\in{\mathcal B}(X^*)\backslash\{\emptyset\}$.

To show the reverse inclusion, let $C$ be a non-empty clopen subset
of $X^*$. Let $C=\mu(Y)$, for some $Y=X\cup\{p\}\subseteq Z$, where
$Z$ is obtained from $\beta X$ by contracting $C$ to a point $p$.
Let $q:\beta X\rightarrow Z$ denote its natural quotient mapping.
Let $U$ be an open subset of $\beta X$ such that $U\cap X^*=C$, and
let $V$ be an open neighborhood of $C$ in $\beta X$ with
$\mbox{cl}_{\beta X}V\subseteq U$. Consider the set $W=(V\cap
X)\cup\{p\}$. Then since  $q^{-1}(W)=V$ is open in $\beta X$, the
set $W$ is an open neighborhood of $p$ in $Y$, and since $W\subseteq
q(\mbox{cl}_{\beta X}V)$, its closure $\mbox{cl}_{Y}W$ is compact.
This shows that $Y$ is locally compact at $p$, and thus
$Y\in{\mathcal T}_K(X)$.
\end{proof}

For a compact zero-dimensional space $X$, the order structure of
${\mathcal B}(X)$ determines the topology of $X$. The following is
now immediate.

\begin{theorem}\label{KHDJ}
For strongly zero-dimensional
locally compact spaces $X$ and $Y$ the following conditions are
equivalent.
\begin{itemize}
\item[\rm(1)] ${\mathcal T}_K(X)$  and  ${\mathcal T}_K(Y)$ are order-isomorphic;
\item[\rm(2)] $X^*$ and $Y^*$  are homeomorphic.
\end{itemize}
\end{theorem}

\begin{theorem}\label{KFDJ}
Let  $X$ be  a locally compact
space and let $Y\in {\mathcal T}(X)$. Then
\[ Y=\bigvee\big\{S\in {\mathcal T}_C(X):Y\geq S\big\}.\]
\end{theorem}

\begin{proof}
This follows from Theorem \ref{DDS} and the fact that
${\mathcal Z}(X^*)$ is a base for closed subsets of
$X^*$.
\end{proof}

The next few result will have applications in the following
sections.

\begin{theorem}\label{AADJ}
Let $X$ be a locally compact
space. Then
\[ \mu\big({\mathcal T}^*(X)\big)=\big\{C\in {\mathcal Z}(\beta X): C\cap
X=\emptyset\big\}\backslash\{\emptyset\}.\]
\end{theorem}

\begin{proof}
Suppose that $Y=X\cup\{p\} \in{\mathcal T}^*(X)$. Let
$\{U_n\}_{n<\omega} $ be a base at $p$ in $Y$ and let $V_n$'s be
open subsets  in $\beta Y$ such that $U_n=V_n\cap Y$. Let for each
$n<\omega$, $f_n\in C( \beta Y, {\bf I})$ be such that $f_n(p)=0$
and $f_n(\beta Y\backslash V_n)\subseteq \{1\}$, and let
$S=\bigcap_{n<\omega} Z(f_n)$. We verify that $S=\{p\}$. For if
$y\in S$ and $y\neq p$, let $U$ and $V$ be disjoint open
neighborhoods of $y$ and $p$ in $\beta Y$, respectively. Let
$k<\omega$ be such that $U_k\subseteq V\cap Y$. Then $\mbox
{cl}_{\beta Y}V_k= \mbox {cl}_{\beta Y}U_k \subseteq \mbox
{cl}_{\beta Y}V$, and therefore since $y\in Z(f_k)$, we have $y\in
\mbox {cl}_{\beta Y}V$.  But this is a contradiction as $y\in U$ and
$U\cap V=\emptyset$. Therefore since $\{p\}\in{\mathcal Z}(\beta
Y)$, we have $\mu(Y)=F_Y^{-1}(p)\in{\mathcal Z}(\beta X)$.

To show the reverse inclusion, let  $\emptyset\neq T\in{\mathcal
Z}(\beta X)$ be such that $T\cap X=\emptyset$, and let $Z$ be the
space obtained from $\beta X $ by contracting $T$ to a point $p$.
Let $Y=X\cup\{p\}\subseteq Z$, and let  $q:\beta X\rightarrow Z$ be
the natural quotient mapping. For each $n<\omega$, let
\[U_n=\big(\big(f^{-1}\big([0,1/n)\big)\backslash T\big)\cup\{p\}\big)\cap Y.\]
Then $U_n$ is an open neighborhood of $p$ in $Y$. Suppose that $U$ is an open neighborhood of $p$ in $Y$ and let $U=V\cap Y$, for some open subset $V$ of $\beta Y$. Then since $p\in V$, we have
\[\bigcap_{n<\omega} f^{-1}\big([0,1/n]\big)=T\subseteq
  q^{-1}(V).\]
Therefore  there exists a $k<\omega$ such that
$f^{-1}([0,1/k])\subseteq q^{-1}(V)$, and thus $U_k\subseteq U$. Therefore $\{U_n\}_{n<\omega}
$ is  a base at $p$ in $Y$  and $T=\mu(Y)\in \mu({\mathcal T}^*(X))$.
\end{proof}

\begin{corollary}\label{AWDJ}
Let $X$ be a locally compact
space. Then ${\mathcal T}^*(X)\neq\emptyset$ if and only if $X$ is
not pseudocompact.
\end{corollary}

\begin{proof}
We note that $\upsilon X$ is the intersection of all
cozero-sets of $\beta X$ which contain $X$. But $X$ is pseudocompact
if and only if $\upsilon X=\beta X$. Now Theorem \ref{AADJ} completes the
proof.
\end{proof}

\begin{theorem}\label{SDJ}
Let $X$ be a locally compact
space. Then
\[ \mu\big({\mathcal T}_P(X)\big)=\big\{C\in{\mathcal C}(X^*): C\supseteq\beta
X\backslash\upsilon X\big\}\backslash\{\emptyset\}.\]
\end{theorem}

\begin{proof}
Suppose that $Y=X\cup\{p\} \in{\mathcal T}_P(X)$. Let
$C=\mu(Y)$. Assume that $(\beta X\backslash\upsilon X)\backslash C
\neq\emptyset$, and let $x\in(\beta X\backslash\upsilon X)\backslash
C$. Since $x\notin C$, there exists an $S\in {\mathcal Z}(\beta X)$
such that $x\in S$ and $S\cap C=\emptyset$. Since $x\notin \upsilon
X$, there exists a $T\in {\mathcal Z}(\beta X)$ such that $x\in T$
and $T\cap X=\emptyset$. Now since $D=(S\cap T)\backslash C$ is a
non-empty $G_\delta$-set of $\beta X$, it is also a non-empty
$G_\delta$-set of $\beta Y$ (which is obtained from $\beta X$ by
contracting $C$ to the point $p$) and therefore by pseudocompactness
of $Y$ we have $D\cap Y\neq \emptyset$.  But this is a contradiction
as $D\cap X= \emptyset$ and $p\notin D$.

To show the reverse inclusion, suppose that $C\in {\mathcal
C}(X^*)\backslash\{\emptyset\}$ and $C\supseteq\beta
X\backslash\upsilon X$. Let $Y=X\cup\{p\}\in {\mathcal T}(X)$ be
such that $\mu(Y)=C$. Suppose that $Y$ is not pseudocompact.  Then
there exists a non-empty $S\in {\mathcal Z}(\beta Y)$ (note that
$\beta Y$ is obtained from $\beta X$ by contracting $C$ to the point
$p$  and
$q:\beta X\rightarrow \beta Y$ is its corresponding quotient mapping)
such that $S\cap Y=\emptyset$. Now $q^{-1}(S)\in {\mathcal Z}(\beta
X)$ and since $p\notin S$, we have $ q^{-1}(S)\cap C=\emptyset$.
Therefore $q^{-1}(S)\subseteq \beta X\backslash C\subseteq\upsilon
X$. Thus since $q^{-1}(S)\in {\mathcal Z}(\upsilon X)$ and
$q^{-1}(S)\neq\emptyset$, we have $q^{-1}(S)\cap X\neq\emptyset$,
which is contradiction,  as $S\cap X=\emptyset$. This shows that $Y$
is pseudocompact, and thus $C\in\mu({\mathcal T}_P(X))$. This
together with the first part of the proof gives the
result.
\end{proof}

\section{The case when $X$ is locally compact and paracompact}

In this section we study the relation between the order structure of
various sets of one-point extensions of a locally compact
paracompact space $X$, and the topology of a certain subspace of
$X^*$. We make use of the following result in a number of occasions
throughout  (see Theorem \ref{YUGG}.27 and 3.8.C of [4]).

\begin{proposition}\label{RAWDJ}
Let $X$ be a locally compact paracompact non-$\sigma$-compact space.
Then we have
\[X=\bigoplus_{i\in I} X_i,\]
where each $X_i$ is a $\sigma$-compact non-compact subspace.
\end{proposition}

Following the notations of [7], for a Tychonoff space $X$, we let
\[\sigma X=\bigcup\{\mbox {cl}_{\beta X}A:A\subseteq X \mbox { is
$\sigma$-compact} \}.\]
Using the notations of Proposition \ref{RAWDJ}, it can be shown
that for a locally compact paracompact non-$\sigma$-compact space
$X$, we have
\[\sigma
X=\bigcup\Big\{\mbox {cl}_{\beta X}\Big(\bigcup_{i\in J}X_i\Big):J\subseteq I
\mbox { is countable}\Big\}\]
which is clearly an open subset of $\beta X$, as each
$\bigcup_{i\in J}X_i$ is clopen in $X$.

Here are some examples showing that neither of the implications,
paracompactness implies local compactness, nor its converse hold.
Clearly the hedgehog with an infinite  number of spines provides an
example of a paracompact space which is not locally compact. Now
consider the space $\sigma X$, when $X$ is an uncountable discrete
space. Then $\sigma X$ is locally compact, as it is open in $\beta
X$.  However, the space $\sigma X$ is not paracompact, as it is
countably compact and non-compact.

The following follows from Theorems 2.5 and 2.8.

\begin{lemma}\label{RAJ}
Let $X$ be a locally compact space. Then
\[\mu\big({\mathcal T}_K^*(X)\big)=\big\{Z\in {\mathcal Z}(\beta X):  \mbox {$Z$
is clopen in $X^*$}\big\}\backslash\{\emptyset\}.\]
\end{lemma}

\begin{lemma}\label{RRAJ}
Let $X$ be a locally compact
paracompact  non-$\sigma$-compact  space and let $Y=X\cup\{p\}\in
{\mathcal T}(X)$. Then the following conditions are equivalent.
\begin{itemize}
\item[\rm(1)] $Y\in {\mathcal T}_K^*(X)$;
\item[\rm(2)] $p$ has a compact neighborhood $U$ in $Y$ such that $U\backslash\{p\}$ is $\sigma$-compact.
\end{itemize}
\end{lemma}

\begin{proof}
{\em (1) implies (2).} Suppose that  $Y\in {\mathcal
T}_K^*(X)$ and let $\{V_n \}_{n< \omega }$  be a base  at $p$ in
$Y$. We may assume that for each $n<\omega $, we have
$V_n\supseteq \mbox {cl}_Y V_{n+1}$ and $\mbox{cl}_Y V_{n}$ is
compact.  Then for each $n<\omega$, the set $\mbox {cl}_Y
V_n\backslash
  V_{n+1}$ is closed in $\mbox {cl}_Y V_1$, and therefore it is
compact. We have $\mbox {cl}_Y
V_1\backslash\{p\}=\bigcup_{n<\omega}(\mbox {cl}_Y V_n\backslash
  V_{n+1})$, and thus  $\mbox {cl}_Y V_1$ is the desired neighborhood
of $p$.

{\em (2) implies (1).} Suppose that $U$ is a compact neighborhood of
$p$ such that  $U\backslash\{p\}$ is $\sigma$-compact. Assume the
notations of Proposition \ref{RAWDJ}. Now since $U\backslash\{p\}$ is
$\sigma$-compact, there exists a countable $J\subseteq I$ such that
$U\backslash\{p\}\subseteq \bigcup_{i\in J} X_i$. By 3.8.C of [4],
for each $i\in J$ we have $X_i=\bigcup_{n<\omega} C^i_n$, where for
each $n<\omega$, the set  $C^i_n$ is  open in $X$ and we have
$\mbox {cl}_X C^i_n\subseteq C^i_{n+1}$ and $\mbox {cl}_X C^i_n$ is
compact. Let
\[\{\mbox {cl}_X C^i_n: i\in J \mbox { and } n<\omega\}=\{D_n
\}_{n<\omega}\]
and consider the family
\[{\mathcal F}=\big\{\mbox {int}_Y U\backslash (D_1\cup\cdots\cup
D_n):n<\omega\big\}\]
of open neighborhoods of $p$ in $Y$. If $V$ is an open
neighborhood of $p$ in $Y$, then since
\[U\subseteq V\cup\bigcup_{i\in J} X_i = V\cup\bigcup \{ C^i_n:i\in
J \mbox{ and } n<\omega \}\]
by compactness of $U$, there exists a $k<\omega$ such
that $U\subseteq V\cup C^{i_1}_{n_1}\cup\cdots\cup C^{i_k}_{n_k}$, and
therefore for some $n<\omega$, $U\backslash (D_1\cup\cdots\cup
D_n)\subseteq V$. This shows that ${\mathcal F}$ is a countable base
at $p$ in $Y$, and since $Y$ is locally compact, it follows that
$Y\in {\mathcal T}_K^*(X)$.
\end{proof}

\begin{lemma}\label{WRRAJ}
For any locally compact
paracompact space $X$, we have ${\mathcal T}_K(X)={\mathcal
T}_K^*(X)$ if and only if $X$ is $\sigma$-compact.
\end{lemma}

\begin{proof}
First suppose that $X$ is  $\sigma$-compact and
let
$Y=X\cup\{p\}\in {\mathcal T}_K(X)$. Let $U$ be an open neighborhood
of $p$ in $Y$ such that $\mbox {cl}_Y U$ is compact. Let
$X=\bigcup_{n<\omega}C_n$, where for each $n<\omega$, the set $C_n$
is  open in $X$ and we have $\mbox {cl}_X C_n\subseteq C_{n+1}$ and
$\mbox {cl}_X C_n$ is compact. We show that the family
$\{U\backslash \mbox {cl}_X C_n:n<\omega\}$ forms a base at $p$ in
$Y$. To show  this, suppose that $V$ is an open neighborhood of $p$
in $Y$. Then since $\{V\}\cup\{C_n:n<\omega\}$ is an open cover of
the compact set $\mbox {cl}_YU$, there exists a $k<\omega$ such that
$\mbox {cl}_YU\subseteq V\cup C_k$. Clearly $U\backslash \mbox
{cl}_XC_k\subseteq V$. This shows that $Y$ is first-countable at
$p$, and thus $Y\in {\mathcal T}_K^*(X)$.

Now suppose  that $X$ is not $\sigma$-compact and assume the
notations of Proposition \ref{RAWDJ}. Let $\omega X=X\cup\{\Omega\}$ be the
one-point compactification of $X$. Clearly $\omega X\in {\mathcal
T}_K(X)$. But since every neighborhood $W$ of $\Omega$ contains all
but a finite number of $X_i$'s, the set $W\backslash\{\Omega\}$ is
not $\sigma$-compact, and thus by Lemma \ref{RRAJ} we have $\omega X\notin
{\mathcal T}_K^*(X)$.
\end{proof}

The next three lemmas are
taken from [8]. We include them in here for the sake of
completeness.

\begin{lemma}\label{WRRJ}
Let $X$ be  a locally compact
paracompact  space. If $\emptyset\neq Z\in {\mathcal Z}(\beta X)$
then $Z\cap \sigma X\neq \emptyset$.
\end{lemma}

\begin{proof}
Suppose that $\{x_n\}_{n<\omega}$ is an infinite
sequence in $\sigma X$. Using the notations of Proposition \ref{RAWDJ},
there exists a countable $J\subseteq I$ such that
$\{x_n\}_{n<\omega}\subseteq \mbox {cl}_{\beta X}(\bigcup_{i\in J}
X_i)$, and therefore $\{x_n\}_{n<\omega}$ has a limit point in
$\sigma X$. Thus $\sigma X$ is countably compact, and therefore
pseudocompact, and $\upsilon (\sigma X)=\beta (\sigma X)=\beta X$.
The result now follows as for any Tychonoff space $T$, any non-empty
zero-set of $\upsilon T$ intersects $T$ (see  Lemma \ref{HGH}(f) of
[12]).
\end{proof}

\begin{lemma}\label{WRORJ}
Let  $X$ be  a locally compact
paracompact space. If $\emptyset\neq Z\in {\mathcal Z}(X^*)$ then
$Z\cap \sigma X\neq \emptyset$.
\end{lemma}

\begin{proof}
Let $S\in Z(\beta X)$ be such that $Z=S\backslash X$.
By the above lemma $S\cap \sigma X\neq\emptyset$. Suppose that
$S\cap (\sigma X\backslash X)=\emptyset$. Then $S\cap \sigma X=S\cap
X$. Let $L=\{i\in I :S\cap X_i\neq\emptyset\}$, where $X_i$'s are as
in Proposition \ref{RAWDJ}. Clearly $L$ is finite. Observe that $\mbox
{cl}_{\beta X}(\bigcup_{i\in L} X_i) $ is clopen in $\beta X$, as
$\bigcup_{i\in L} X_i$ is clopen in $X$. Let $f$ be its
characteristic function which is in $C^*(X)$.  Now since $Z(f)\cap
S\in {\mathcal Z}(\beta X)$ misses $\sigma X$, by the above lemma,
$Z(f)\cap S=\emptyset$. But since $\beta X\backslash \sigma
X\subseteq Z(f)$, we have $Z=S\cap(\beta X\backslash \sigma
X)\subseteq S\cap Z(f)=\emptyset$, which is a contradiction.
Therefore $Z\cap (\sigma X\backslash X)=S\cap (\sigma X\backslash
X)\neq\emptyset$.
\end{proof}

\begin{lemma}\label{WRPJ}
Let $X$ be  a locally compact
paracompact  space and let $ S, T\in {\mathcal Z}(X^*)$. If $S\cap
\sigma X\subseteq T\cap \sigma X$ then $S\subseteq T$.
\end{lemma}

\begin{proof}
Suppose that $S\backslash T\neq\emptyset$. Let $x\in
S\backslash T$. Let $f\in C(\beta X,{\bf I})$  be such that $f(x)=0$
and $f(T)\subseteq\{1\}$. Then $Z(f)\cap S\in {\mathcal Z}(X^*)$ is
non-empty, and therefore by the above lemma, $Z(f)\cap S\cap \sigma
X\neq\emptyset$. But this is impossible as  $Z(f)\cap S\cap \sigma
X\subseteq Z(f)\cap T=\emptyset$.
\end{proof}

\begin{lemma}\label{RWRPJ}
Let  $X$ be a locally compact
paracompact  space. If $Y\in  {\mathcal T}^*_K(X)$ then
$\mu(Y)\subseteq  \sigma X$.
\end{lemma}

\begin{proof}
Let $C=\mu (Y)$, for some $Y=X\cup\{p\}\in{\mathcal
T}^*_K(X)$. By Lemma \ref{RRAJ}, there exists a compact neighborhood $W$ of
$p$ in $Y$ such that $W\backslash\{p\}$ is $\sigma$-compact. We
claim that $F_Y^{-1}(p)\subseteq \mbox {cl}_{\beta
X}(W\backslash\{p\})$. So suppose to the contrary that there exists
an $x\in F_Y^{-1}(p)$ such that  $ x\notin \mbox {cl}_{\beta
X}(W\backslash\{p\})$. Let $U$ be an open neighborhood of $x$ in
$\beta X$ which misses $W\backslash\{p\}$. Since $Y$ is locally
compact, $W$ is also a neighborhood of $p$ in $\beta Y$, and
therefore, there exists an open neighborhood  $V$ of $x$ in $\beta
X$ such that $F_Y(V)\subseteq W$. If $t\in U\cap V\cap X$, then
$t=F_Y(t)\in U\cap W$, which is a contradiction. Therefore
$C=F_Y^{-1}(p)\subseteq  \mbox {cl}_{\beta
X}(W\backslash\{p\})\subseteq \sigma
X$.
\end{proof}

The proof of the following is a modification of  the ones we have
given for Theorems \ref{YROEJ} and \ref{YPPJ} of [8]. Note that a space $X$ is
locally compact and $\sigma$-compact if and only if $X^*\in
{\mathcal Z}(\beta X)$ (see 1B of [13]). We use this fact in several
different places.

\begin{theorem}\label{PSDJ}
For  zero-dimensional locally
compact paracompact spaces $X$ and $Y$ the following conditions  are
equivalent.
\begin{itemize}
\item[\rm(1)] ${\mathcal T}^*_K(X)$ and  ${\mathcal T}^*_K(Y)$ are order-isomorphic;
\item[\rm(2)] $\sigma X\backslash X$ and $\sigma Y\backslash Y$ are homeomorphic.
\end{itemize}
\end{theorem}

\begin{proof}
{\em (1) implies (2).} Suppose that condition  (1) holds.
Assume that  one of $X$ and $Y$, say $X$, is $\sigma$-compact.
Suppose that $Y$ is not $\sigma$-compact and let $Y=\bigoplus_{i\in
J}Y_i$, with $Y_i$'s being $\sigma$-compact non-compact subspaces.
Since by Lemma \ref{WRRAJ} we have ${\mathcal T}^*_K(X)={\mathcal T}_K(X)$,
and ${\mathcal T}_K(X)$ has a minimum, namely its one-point
compactification,  ${\mathcal T}^*_K(X)$ and thus ${\mathcal
T}^*_K(Y)$ has a minimum. Let $T$ be the minimum of ${\mathcal
T}^*_K(Y)$. Then since for each countable $L\subseteq J$, we have
$(\bigcup_{i\in L}Y_i )^*\in \mu_Y({\mathcal T}^*_K(Y))$, it follows
that  $(\bigcup_{i\in L}Y_i )^*\subseteq \mu_Y(T)$, and thus $\sigma
Y\backslash Y\subseteq\mu_Y(T)$. Now by Lemma \ref{WRPJ}, with $Y^*$ and
$\mu_Y(T)$ being the zero-sets, we have $Y^*\subseteq\mu_Y(T)$. But
by Lemma \ref{RAJ}, we have $\mu_Y(T)\in  {\mathcal Z}(\beta Y)$, and
therefore $Y^*\in {\mathcal Z}(\beta Y)$, which is a contradiction,
as we assumed that $Y$ is not $\sigma$-compact (see 1B of [13]).
Thus $X$ and $Y$ are both $\sigma$-compact, and so by Lemma \ref{WRRAJ} and
condition (1), ${\mathcal T}_K(X)$ and ${\mathcal T}_K(Y)$ are
order-isomorphic. Thus since $X$ and $Y$ are zero-dimensional
locally compact paracompact, each is  strongly zero-dimensional (see
Theorem \ref{EEFR}.10 of [4]). Now  Theorem \ref{KHDJ} implies that $\sigma
X\backslash X=X^*$ and $\sigma Y\backslash Y=Y^*$ are homeomorphic.

Next suppose that $X$ and
$Y$ are both non-$\sigma$-compact and let  $\phi:{\mathcal
T}^*_K(X)\rightarrow{\mathcal T}^*_K(Y)$ be an order-isomorphism.
Let $g=\mu_Y \phi \mu_X^{-1}:\mu_X({\mathcal
T}^*_K(X))\rightarrow\mu_Y({\mathcal T}^*_K(Y))$ and let
$\omega\sigma X=\sigma X\cup\{\Omega\}$ and $\omega\sigma Y=\sigma
Y\cup\{\Omega '\}$ be one-point compactifications.  We define a
function $ G:{\mathcal B}(\omega\sigma X\backslash
X)\rightarrow{\mathcal B}(\omega\sigma Y\backslash Y)$ between the
two Boolean algebras of clopen sets, and verify that it is an
order-isomorphism.

Set $G(\emptyset)=\emptyset$ and $G(\omega\sigma X\backslash
X)=\omega\sigma Y\backslash Y$. Let $U\in {\mathcal B}(\omega\sigma
X\backslash X)$. If $U\neq\emptyset$ and $\Omega\notin U$, then $U$
is an open subset of $\sigma X\backslash X$, and therefore it is an
open subset of $X^*$. There exists a countable $J\subseteq I$ such
that $U\subseteq (\bigcup_{i\in J}X_i)^*$, where $X=\bigoplus_{i\in
I} X_i$, with $X_i$'s being $\sigma$-compact non-compact subspaces,
and thus $U\in \mu_X ({\mathcal T}^*_K(X))$. In this case we let $
G(U)= g(U)$. If $U\neq\omega\sigma X\backslash X$ and $\Omega\in U$,
then $(\omega\sigma X\backslash X)\backslash U\in \mu_X({\mathcal
T}^*_K(X))$, and we let $ G(U)=(\omega\sigma Y\backslash
Y)\backslash g((\omega\sigma X\backslash X)\backslash U)$.

To show that $G$ is an order-isomorphism, let $U,V\in{\mathcal
B}(\omega\sigma X\backslash X)$ with $U\subseteq V$. We may assume
that  $U\neq\emptyset$ and $V\neq\omega\sigma X\backslash X$. We
consider the following cases.

{\em Case 1)} Suppose that $\Omega\notin V$. Then clearly
$G(U)=g(U)\subseteq g(V)=G(V)$.

{\em Case 2)} Suppose that  $\Omega\notin U$ and $\Omega\in V$. If
$G(U)\backslash G(V)\neq\emptyset$, then $T=g(U)\cap g((\omega\sigma
X\backslash X)\backslash V)\neq\emptyset$, and therefore by Lemma \ref{RAJ}, we have  $T\in\mu_Y({\mathcal T}^*_K(Y))$. Let $S\in
\mu_X({\mathcal T}^*_K(X))$ be such that $g(S)=T$. Then since $g$ is
an order-isomorphism, $S\subseteq U\cap((\omega\sigma X\backslash
X)\backslash V)=\emptyset$, which is a contradiction. Therefore
$G(U)\subseteq G(V)$.

{\em Case 3)} Suppose that $\Omega\in U$. Then since $(\omega\sigma
X\backslash X)\backslash V\subseteq(\omega\sigma X\backslash
X)\backslash U$ we have
\[G(U)=(\omega\sigma Y\backslash Y)\backslash g\big((\omega\sigma
X\backslash X)\backslash U\big)\subseteq (\omega\sigma Y\backslash
Y)\backslash g\big((\omega\sigma X\backslash X)\backslash V\big)=G(V).\]
This shows that $G$ is an order-homomorphism.

To complete the proof we note that since $\phi^{-1}:{\mathcal
T}^*_K(Y)\rightarrow {\mathcal T}^*_K(X)$ also is  an
order-isomorphism, if we denote $h=\mu_X \phi^{-1}\mu_Y^{-1}$, then
arguing as above, $h$ induces an order-homomorphism $H:{\mathcal
B}(\omega\sigma Y\backslash Y)\rightarrow{\mathcal B}(\omega\sigma
X\backslash X)$. It is then easy to see that $H=G^{-1}$.

Now since by Theorem \ref{EEFR}.10 of [4] the spaces $X$ and $Y$ are
strongly zero-dimensional, $\sigma X$ and $\sigma Y$, and therefore
their one-point compactifications $\omega\sigma X$ and $\omega\sigma
Y$, also are  zero-dimensional. Thus by Stone Duality, there exists
a homeomorphism $f:\omega\sigma X\backslash X\rightarrow
\omega\sigma Y\backslash Y$ such that $f(U)=G(U)$, for every
$U\in{\mathcal B}(\omega\sigma X\backslash X)$. Now since for every
countable $J\subseteq I$, we have $\Omega'\notin g(Q_J)=
G(Q_J)=f(Q_J)$, where $Q_J=(\bigcup_{i\in J} X_i)^*$, the function
$f|(\sigma X\backslash X):\sigma X\backslash X\rightarrow\sigma
Y\backslash Y$ is a homeomorphism.

{\em (2) implies (1).} Suppose that condition (2) holds. If one of $X$
and $Y$, say $X$, is $\sigma$-compact, then
since $\sigma Y\backslash Y$ and $X^*=\sigma X\backslash X$ are
homeomorphic, $\sigma Y\backslash Y$
  is compact. Suppose that $Y$ is not $\sigma$-compact and let $Y_i$'s
be as in the previous part.
  By compactness of $\sigma Y\backslash Y$, there exists a countable
$L\subseteq J$ such that
  $(\bigcup_{i\in L}Y_i)^*=\sigma Y\backslash Y$, which is clearly
false. Thus $Y$
    also is $\sigma$-compact, and since $X^*$ and $Y^*$ are
homeomorphic, by Theorem \ref{KHDJ} and Lemma \ref{WRRAJ}
    we have  that  ${\mathcal T}^*_K(X)$ and ${\mathcal T}^*_K(Y)$ are
    order-isomorphic.

Next suppose that $X$ and $Y$ are both non-$\sigma$-compact  and
let  $f:\sigma X\backslash X\rightarrow \sigma Y\backslash Y$ be a
homeomorphism. Let $Z\in\mu_X({\mathcal T}^*_K(X))$. Then by Lemma \ref{RWRPJ}, we have $Z\subseteq \sigma X\backslash X$, and thus  there
exists a countable $A\subseteq I$ such that $Z\subseteq \mbox
{cl}_{\beta X} P$, where $P=\bigcup_{i\in A} X_i$.  But since $P^*$
is clopen in $\sigma X\backslash X$, $f(P^*)$ is clopen in $\sigma
Y\backslash Y$, and since it is also compact, there exists a
countable $B\subseteq J$ such that $f(P^*)\subseteq Q^*$, where
$Q=\bigcup_{i\in B} Y_i$ and  $Y=\bigoplus_{i\in J} Y_i$, with each
$Y_i$ being a $\sigma$-compact non-compact subspace. Since by Lemma \ref{RAJ}, the set  $Z$ is clopen in $X^*$, the set  $f(Z)$ is clopen in
$\sigma Y\backslash Y$, and since we have $f(Z)\subseteq Q^*$, it
also is  clopen in $ Q^*$, and thus clopen in $Y^*$, i.e., $f(Z)\in
\mu_Y ({\mathcal T}^*_K(Y))$. Now we define a function $F: \mu_X
({\mathcal T}^*_K(X))\rightarrow\mu_Y ({\mathcal T}^*_K(Y))$ by
$F(Z)=f(Z)$. The function $F$ is clearly well-defined and it is an
order-homomorphism. Since $f^{-1}$ is also a homeomorphism, arguing
as above, we can define a function $G: \mu_Y ({\mathcal
T}^*_K(Y))\rightarrow\mu_X ({\mathcal T}^*_K(X))$ by
$G(Z)=f^{-1}(Z)$, which is clearly the inverse of $F$. Thus $F$ is
an order-isomorphism.
\end{proof}

The following question naturally arises in connection with Theorem \ref{PSDJ} above.

\begin{question}\label{UPSDJ}
Is there any subset of $X^*$ whose
topology determines and is determined by the order structure of
${\mathcal T}^*(X)$? (See Theorem \ref{HJL} for a partial answer to this
question)
\end{question}

We note that for a locally compact space $X$, each  Lindel\"{o}f
subspace  of $X$ is a subset of a $\sigma$-compact  subset of $X$,
and therefore we can describe the elements of ${\mathcal T}_D(X)$ as
those $Y=X\cup\{p\}\in {\mathcal T}(X)$ for which $p\notin \mbox
{cl}_Y A$, for any $\sigma$-compact $A\subseteq X$.

\begin{theorem}\label{EPSDJ}
Let $X$ be a locally compact
paracompact  space. Then
\[\mu\big({\mathcal T}_D(X)\big)={\mathcal C}(\beta X\backslash \sigma
X)\backslash\{\emptyset\}.\]
\end{theorem}

\begin{proof}
Suppose that $Y=X\cup\{p\}\in {\mathcal T}_D(X)$ and
let $C=F_Y^{-1}(p)$. Assume that $C\cap\sigma X\neq\emptyset$ and
let $x\in C\cap\sigma X$. Let $A$ be a $\sigma$-compact subset of
$X$ such that $x\in \mbox {cl}_{\beta X}A$. By assumption $p\notin
\mbox {cl}_Y A$. Therefore $U\cap Y\cap A=\emptyset$,  for some open
neighborhood $U$ of $p$ in $\beta Y$. Now since $F_Y^{-1}(U)$ is an
open neighborhood of $x$ in $\beta X$, we have $A\cap
F_Y^{-1}(U)\neq\emptyset$. Let $a\in A\cap F_Y^{-1}(U)$. Then
$a=F_Y(a)\in U\cap A$, which is a contradiction.
Therefore $\mu(Y)=C\subseteq\beta X\backslash \sigma X$.

Conversely, suppose that  $C\in{\mathcal C}(\beta X\backslash \sigma X)\backslash\{\emptyset\}$. Then since $\sigma X$ is open in $\beta X$, we have $C\in{\mathcal C}(X^*)$, and thus $C=\mu (Y)$, for some $Y=X\cup\{p\}\in {\mathcal T}(X)$. Suppose that $A\subseteq X$ is $\sigma$-compact. Then since $\mbox {cl}_{\beta X}A\subseteq \sigma X$, we have  $C\cap\mbox {cl}_{\beta X}A=\emptyset$. Let $U=(\beta X\backslash (C\cup\mbox {cl}_{\beta X}A))\cup\{p\}$. Then $U\cap Y$ is an open
neighborhood of $p$ in $Y$, and $U\cap Y\cap A=\emptyset$. Therefore $p\notin\mbox{cl}_Y A$, which shows that $Y\in {\mathcal T}_D(X)$.
\end{proof}

\begin{lemma}\label{ERWRPJ}
Suppose that  $X$ is a locally
compact paracompact space and let $Y\in {\mathcal T}(X)$. Then $Y\in
{\mathcal T}_L(X)$ if and only if $\mu(Y)\supseteq \beta
X\backslash\sigma X$.
\end{lemma}

\begin{proof}
Clearly it suffices to consider only the case when $X$ is
non-$\sigma$-compact. Suppose that $Y\in {\mathcal T}_L(X)$ and let
$C=\mu(Y)$. Assume that $ (\beta X\backslash\sigma X)\backslash
C\neq\emptyset$ and let $ x\in (\beta X\backslash\sigma X)\backslash
C\neq\emptyset$. Let $U$ and $V$ be disjoint open neighborhoods of
$x$ and $C$ in $\beta X$, respectively. Assume the notations of
Proposition \ref{RAWDJ} and let
\[J=\{i\in I :X_i\cap U\neq\emptyset\}.\]
Clearly
\[\mbox {cl}_{\beta X}U=\mbox {cl}_{\beta X}(U\cap X)\subseteq
\mbox {cl}_{\beta X}\Big(\bigcup_{i\in J}X_i\Big)\]
and thus since $x\notin \sigma X$, the set $J$ is
uncountable. Then $X\backslash V$, being closed in the Lindel\"{o}f
space $Y$, is Lindel\"{o}f. But this is a contradiction, as since
$U$ intersects uncountably many of $X_i$'s, there  is no countable
subcover of $\{X_i\}_{i\in I}$ covering $X\backslash V$. Therefore
$C\supseteq \beta X\backslash\sigma X$.

To prove the converse, suppose that  $\mu(Y)=C\supseteq \beta
X\backslash\sigma X$. Let ${\mathcal V}$ be an open cover of
$Y=X\cup\{p\}$. Let $V\in{\mathcal V}$ be such that $p\in V$, and
let $W$ be an open set in $\beta Y$ such that $V=W\cap Y$. Then
since $p\in W$, we have  $ \beta X\backslash F_Y^{-1}(W)\subseteq
\sigma X$, and therefore $\beta X\backslash F_Y^{-1}(W)\subseteq
\mbox {cl}_{\beta X}M$, where $M=\bigcup_{i\in J}X_i$ and
$J\subseteq I$ is countable. Clearly $Y\backslash V\subseteq M$. But
$M$, being $\sigma$-compact, can be covered by countably many
subsets of ${\mathcal V}$. Therefore ${\mathcal V}$ has a countable
subcover, which shows that  $Y$ is
Lindel\"{o}f.
\end{proof}

\begin{theorem}\label{EJ}
Let $X$ be a locally compact
paracompact  non-Lindel\"{o}f space. Then the minimum of ${\mathcal
T}_D(X)$ is the unique Lindel\"{o}f element of ${\mathcal T}_D(X)$.
\end{theorem}

\begin{proof}
Let $Y\in{\mathcal T}(X)$ be such that $\mu (Y)=\beta
X\backslash\sigma X$. Then by Theorem \ref{EPSDJ} we have  $Y\in{\mathcal
T}_D(X)$, and by Lemma \ref{ERWRPJ} the space $Y$ is Lindel\"{o}f. Suppose
that $S\in {\mathcal T}_D(X)$ is Lindel\"{o}f. Then by the above
lemma we have $\mu (S)\supseteq \beta X\backslash \sigma X$, and
thus $S=Y$.
\end{proof}

\begin{theorem}\label{OEJ}
For  locally compact paracompact
spaces $X$ and $Y$ the following conditions are
equivalent.
\begin{itemize}
\item[\rm(1)] ${\mathcal T}_D(X)$ and ${\mathcal T}_D(Y)$ are order-isomorphic;
\item[\rm(2)] $\beta X\backslash\sigma X$ and $\beta Y\backslash\sigma Y$ are
homeomorphic.
\end{itemize}
\end{theorem}

For a Tychonoff space $X$ and a cardinal number $\alpha$, let
${\mathcal T}_\alpha(X)$ consist of exactly those $Y=X\cup\{p\}\in
{\mathcal T}(X)$ such that $p\notin \mbox {cl}_Y(\bigcup{\mathcal
F})$, for any discrete family ${\mathcal F}$ of compact open subsets
of $X$ with $|{\mathcal F}|=\alpha$. We also let $\tau_\alpha X$ denote the set
\[\!\!\!\!\!\!\!\!X\cup\bigcup\Big\{\mbox {cl}_{\beta X}\Big(\bigcup{\mathcal
F}\Big):\\{\mathcal F} \mbox { is a discrete family of compact open
subsets of $X$ with $|{\mathcal F}|=\alpha$}\Big\}.\]
Clearly, when $X$ is locally compact, $\tau_\alpha X $ is
an open subset of $\beta X$. If $X$ is a zero-dimensional locally
compact paracompact non-$\sigma$-compact space, then $\sigma
X=\tau_\omega X$. This is because, assuming the notations of
Proposition \ref{RAWDJ},  by 3.8.C of [4], for each $i\in I$ we have
$X_i=\bigcup_{n<\omega} C_n^i$, where each $C_n^i$ is open in $X$,
such that $\mbox {cl}_X C_n^i\subseteq C_{n+1}^i$ and $\mbox {cl}_X
C_n^i$ is compact.
Let  $D_n^i$ be a clopen subset of $X$ such that $ C_n^i\subseteq
D_n^i\subseteq C_{n+1}^i$. Then since each $X_i=D_1^i\oplus\bigoplus_{n\geq 1}(D_{n+1}^i\backslash D^i_n)$, the space  $X$ is a sum of compact open subsets, and thus
$\tau_\omega X=\sigma X$.

Now using the same proof as the one we applied for Theorem \ref{EPSDJ} we obtain the
following result.

\begin{theorem}\label{PPOEJ}
et $X$ be a locally compact
space and let $\alpha$ be a cardinal number. Then
\[\mu\big({\mathcal T}_\alpha(X)\big)={\mathcal C}(\beta X\backslash
\tau_\alpha X)\backslash\{\emptyset\}.\]
\end{theorem}

\begin{theorem}\label{PPJ}
For  locally compact spaces $X$
and $Y$ and a cardinal number $\alpha$  the following conditions are
equivalent.
\begin{itemize}
\item[\rm(1)] ${\mathcal T}_\alpha(X)$ and ${\mathcal T}_\alpha(Y)$ are order-isomorphic;
\item[\rm(2)] $\beta X\backslash \tau_\alpha X$ and $\beta Y\backslash\tau_\alpha Y$ are homeomorphic.
\end{itemize}
\end{theorem}

\begin{lemma}\label{EPPJ}
Let $X$ be  a locally compact
paracompact space. Then
\[ \mu\big({\mathcal T}_S(X)\big)=\big\{C\in {\mathcal C}(X^*):C\subseteq \sigma
X\big\}\backslash\{\emptyset\}.\]
\end{lemma}

\begin{proof}
Let $Y=X\cup\{p\}\in{\mathcal T}(X)$ and let
$C=\mu(Y)$. Let $Z$ be the space obtained from $\beta X$ by
contracting $C$ to the point  $p$, and let $q:\beta X\rightarrow
Z=\beta Y$ be its natural quotient mapping.  Suppose that
$Y=X\cup\{p\}\in{\mathcal T}_S(X)$. Let $U$ be a neighborhood of $p$
in $Y$ with
$U\backslash\{p\} $ being $\sigma$-compact. Then using the notations
of Proposition \ref{RAWDJ}, we have
$U\backslash\{p\}\subseteq G$, where $G=\bigcup_{i\in J} X_i$ and
$J\subseteq I$ is
countable. We verify that $C\subseteq G^*$. So suppose to the contrary
that there exists
an $x\in C\backslash G^*$. Let $V$ be an open neighborhood of $p$ in
$\beta Y$ with $ V\cap Y\subseteq
U$. Then since $p\in V$, we have $C\subseteq q^{-1}(V)$. Now since
$x\in H=q^{-1}(V)\backslash
\mbox {cl}_{\beta X}G$, we have $H\cap X\neq\emptyset$. If $t\in H\cap
X$, then
$t=q(t)\in V\cap X\subseteq U\backslash\{p\}\subseteq G$, which is a
contradiction, as $G\cap H=\emptyset$.
This shows that $C\subseteq G^*\subseteq \sigma X$.

Conversely,
suppose that $C\in {\mathcal C}(X^*)\backslash\{\emptyset\}$ is such
that $C\subseteq \sigma
X$, and let $\mu(Y)=C$, for some $Y=X\cup\{p\}\in{\mathcal T}(X)$.
Let $G=\bigcup_{i\in J} X_i$ be such that $C\subseteq \mbox
{cl}_{\beta X} G$, where $J\subseteq I$ is
countable. Let $U=((\mbox {cl}_{\beta X} G\backslash C)\cup\{p\})\cap
Y$.
Then $U$ is an open neighborhood of $p$ in $Y$ such that
$U\backslash\{p\}=G$ is $\sigma$-compact. Therefore $Y\in{\mathcal
T}_S(X)$.
\end{proof}

\begin{lemma}\label{UPPJ}
Let $X$ be a locally compact
paracompact  space.  Then
\[{\mathcal T}^*_S(X)=\big\{Y=X\cup\{p\}\in {\mathcal T}^*(X):p \mbox { has
a $\sigma$-compact neighborhood in } Y\big\}\]
and
\[\mu\big({\mathcal T}^*_S(X)\big)=\big\{C\in{\mathcal Z}(\beta X):
C\subseteq\sigma X\backslash X \big\}\backslash\{\emptyset\}.\]
\end{lemma}

\begin{proof}
Let ${\mathcal T}$ denote the set of all
$Y=X\cup\{p\}\in {\mathcal T}^*(X) $ such that $p$ has a
$\sigma$-compact neighborhood in  $Y$. First we find $\mu({\mathcal
T})$.  Suppose that  $Y\in{\mathcal T}$ and let $C=\mu(Y)$.
By Theorem \ref{AADJ}, we have  $C\in {\mathcal Z}(\beta X)$. Suppose that
$\{U_n\}_{n<\omega}$ is a base at $p$ in $Y$. We may assume that
$U_1\supseteq U_2\supseteq\cdots$. For each $n<\omega$, let
$U_n=V_n\cup\{p\}$. Then since $p$ has a $\sigma$-compact
neighborhood in $Y$,  there exists a $k<\omega$ such that $\mbox
{cl}_YU_k $ is
$\sigma$-compact. We have
\[\mbox {cl}_XV_k=\mbox {cl}_YU_k\backslash\{p\}=\bigcup_{n\geq
k}\big((\mbox {cl}_YU_n)\backslash
U_{n+1}\big)\]
where for each $n\geq k$, the set  $\mbox
{cl}_YU_n\backslash
U_{n+1}$, being a closed subset of $\mbox {cl}_YU_k$, is
$\sigma$-compact.
Clearly  $\mbox {cl}_{\beta X}(\mbox {cl}_XV_k)\subseteq\sigma X$. We
verify that
$C\subseteq \mbox {cl}_{\beta X}(\mbox {cl}_XV_k)$. To this end, let
$Z=\beta Y$ be the space obtained from $\beta X$ by
contracting $C$ to the point  $p$, and let $q:\beta X\rightarrow \beta
Y$ be
its  natural quotient mapping. Suppose that  $x\in C$ is such that
$x\notin \mbox {cl}_{\beta X}(\mbox {cl}_XV_k)$, and let $U$ be an
open neighborhood of
$x$ in $\beta X$ such that $U\cap \mbox {cl}_XV_k=\emptyset$.  Let $V$
be
an open set in $\beta Y$ such that $U_k=V\cap Y$, and let $W$ be an
open neighborhood of $x$ in $\beta X$
  with $q(W)\subseteq V$. Let $y\in U\cap W\cap X$.  Then $y=q(y)\in
V$, and thus $y\in U_k\cap X=V_k$. But this is a contradiction as
  $y\in U$. This shows that $C\subseteq \sigma X$. Thus $\mu({\mathcal
T})\subseteq\{C\in  {\mathcal Z}(\beta X): C\subseteq\sigma
X\backslash X \}\backslash\{\emptyset\}$.

Next suppose that $C\in{\mathcal Z}(\beta X)\backslash\{\emptyset\}$
is such that  $C\subseteq \sigma X\backslash X$. Let $C=\mu(Y)$, for
some $Y\in{\mathcal T}^*(X)$. Since $C\subseteq \sigma X$, using the
notations of Proposition \ref{RAWDJ}, there exists a countable $J\subseteq
I$ such that $C\subseteq \mbox {cl}_{\beta X}M$, where
$M=\bigcup_{i\in J} X_i$. Since $A=(\mbox {cl}_{\beta X}M\backslash
C)\cup\{p\}$ is open in $\beta Y$, as $\beta Y$ is the quotient
space of $\beta X$ obtained by contracting $C$ to the point $p$, the
set $M\cup\{p\}=A\cap Y$ is a $\sigma$-compact open neighborhood of
$p$ in $Y$. Therefore $ Y\in{\mathcal T}$, and thus $C\in \mu(
{\mathcal T})$.

To complete the proof we note that combining Theorem \ref{AADJ} and Lemma
\ref{EPPJ} we have $\mu({\mathcal T}^*_S(X))=\mu({\mathcal
T}^*(X))\cap\mu({\mathcal T}_S(X))=\mu({\mathcal T})$, from  which
it follows that ${\mathcal T}^*_S(X) ={\mathcal
T}$.
\end{proof}

In the following we show that the order structure of ${\mathcal
T}^*_S(X)$ can determine the topology of the set $\sigma X\backslash
X$. The proof is a slight modification of the  metric case  we gave
in Theorem \ref{HJL} of [8].

\begin{theorem}\label{ROEJ}
For locally compact paracompact
spaces $X$ and $Y$ the following conditions  are equivalent.
\begin{itemize}
\item[\rm(1)] ${\mathcal T}^*_S(X)$ and ${\mathcal T}^*_S(Y)$ are order-isomorphic;
\item[\rm(2)] $\sigma X\backslash X$ and $\sigma Y\backslash Y$ are homeomorphic.
\end{itemize}
\end{theorem}

\begin{proof}
{\em (1) implies  (2).} Suppose that  only one of $X$
and $Y$, say $X$, is $\sigma$-compact (non-compact). Then clearly
${\mathcal T}^*_S(X)={\mathcal T}^*(X)$. Since by 1B of [13] we have
$X^*\in {\mathcal Z}(\beta X)$, the set  ${\mathcal T}^*_S(X)$ has a
minimum, namely $\omega X$.  Now using the same line of reasoning as
in Theorem \ref{PSDJ} ((1) implies (2)) we get a contradiction, which shows
that $Y$  also is $\sigma$-compact. Therefore since ${\mathcal
T}_C(X)$ and ${\mathcal T}_C(Y)$ are order-isomorphic, by Theorem \ref{TDJ}, the spaces $\sigma X\backslash X=X^*$ and $\sigma Y\backslash
Y=Y^*$ are homeomorphic.

Next suppose that  $X$ and $Y$ are both non-$\sigma$-compact. Since
$\mu_X$ and $\mu_Y$ are both order-anti-isomorphism, by condition
(1), there exists an order-isomorphism $\phi: \mu_X ({\mathcal
T}^*_S(X))\rightarrow\mu _Y({\mathcal T}^*_S(Y))$. We extend $\phi$
by letting $\phi(\emptyset)=\emptyset$. Let $\omega\sigma X=\sigma
X\cup\{\Omega\}$ and $\omega\sigma Y=\sigma Y\cup\{\Omega'\}$ be
one-point compactifications. We define a function
\[ \psi :{\mathcal Z}(\omega\sigma X\backslash
X)\rightarrow{\mathcal Z}(\omega\sigma Y\backslash Y)\]
and verify that it is an order-isomorphism.

For a
$Z\in {\mathcal Z}(\omega\sigma X\backslash X)$, with $\Omega\notin
Z$, assuming the notations of Proposition \ref{RAWDJ}, since $Z\subseteq
\mbox {cl}_{\beta X} (\bigcup_{i\in K}X_i)$, for some countable
$K\subseteq I$, we have $Z\in {\mathcal Z}(\beta X)$, and therefore
$Z\in \mu_X ({\mathcal T}^*_S(X))\cup\{\emptyset\}$. In this case,
let  $\psi(Z)=\phi(Z)$.

Now suppose that $Z\in {\mathcal Z}(\omega\sigma X\backslash X)$ and
$\Omega\in Z$. Then $(\omega\sigma X\backslash X)\backslash Z$,
being a cozero-set in $\omega\sigma X\backslash X$, can be written
as
\[(\omega\sigma X\backslash X)\backslash Z=\bigcup_{n<\omega}Z_n\]
where for each $n<\omega$, we have  $ Z_n\in {\mathcal
Z}(\omega\sigma X\backslash X)$ and $\Omega\notin Z_n$, and thus
$Z_n\in\mu_X ({\mathcal T}^*_S(X))\cup\{\emptyset\}$. We claim that
$ \bigcup_{n<\omega}\phi (Z_n)$ is a cozero-set in $\omega\sigma
Y\backslash Y$.

To show this, let $Y=\bigoplus_{i\in J} Y_i$, with each $Y_i$ being
a  $\sigma$-compact non-compact subspace. Since for each $n<\omega$,
we have $\phi(Z_n)\subseteq\sigma Y\backslash Y$, there exists a
countable $L\subseteq J$ such that
\[\bigcup_{n<\omega}\phi (Z_n)\subseteq\Big(\bigcup_{i\in
L}Y_i\Big)^*=\phi(A)\]
for some $A\in \mu_X ({\mathcal T}^*_S(X))$. We show that
\[\phi(A\cap Z)=\phi(A)\backslash\bigcup_{n<\omega}\phi (Z_n).\]
Since for each $n<\omega$, we have $A\cap Z\cap
Z_n=\emptyset$, it follows that  $\phi(A\cap Z)\cap
\phi(Z_n)=\emptyset$, and therefore $\phi(A\cap
Z)\subseteq\phi(A)\backslash\bigcup_{n<\omega}\phi (Z_n)$.

To show
the converse, let $x\in\phi(A)\backslash\bigcup_{n<\omega}\phi
(Z_n)$. Since for each $n<\omega$, we have  $x\notin \phi (Z_n)$,
there exists a $B\in {\mathcal Z}(\omega\sigma Y\backslash Y)$ such
that $x\in B$, and for each $n<\omega$, we have $B\cap\phi
(Z_n)=\emptyset$. If $x\notin\phi(A\cap Z)$, then there exists a
$C\in {\mathcal Z}(\omega\sigma Y\backslash Y)$ such that $x\in C$
and $C\cap\phi(A\cap Z)=\emptyset$.  Consider $D=\phi(A)\cap B\cap
C\in\mu_Y ({\mathcal T}^*_S(Y))$, and let $E\in\mu_X ({\mathcal
T}^*_S(X))$ be such that $\phi (E)=D$. Then since $\phi(E)\cap
\phi(Z_n)=\emptyset$, for each $n<\omega$ we have $E\cap
Z_n=\emptyset$, and therefore $ E\subseteq Z$. On the other hand
since $\phi(E)\subseteq\phi(A)$, we have $E\subseteq A$ and thus
$E\subseteq A\cap Z$. Therefore $\phi (E)\subseteq \phi (A\cap Z)$,
which implies that $\phi(E)=\emptyset$, as $\phi(E)\subseteq C$.
This contradiction shows that $x\in\phi (A\cap Z)$, and therefore
$\phi(A\cap Z)=\phi(A)\backslash\bigcup_{n<\omega}\phi (Z_n)$.

Now
since $\phi(A)$ is  clopen in $\sigma Y\backslash Y$, by definition
of $A$,  we have
\begin{eqnarray*}
(\omega\sigma Y\backslash Y)\backslash\bigcup_{n<\omega}\phi (Z_n)&=&
\Big(\phi (A)\backslash\bigcup_{n<\omega}\phi (Z_n)\Big)\cup \big((\omega\sigma
Y\backslash Y)\backslash \phi (A)\big)\\&=& \phi (A\cap
Z)\cup\big((\omega\sigma Y\backslash Y)\backslash \phi(A)\big)\in{\mathcal
Z}(\omega\sigma Y\backslash Y)
\end{eqnarray*}
and our claim is verified. In this case we define
\[ \psi (Z)=(\omega\sigma Y\backslash
Y)\backslash\bigcup_{n<\omega}\phi (Z_n).\]

Next we show that $\psi$
is well defined. So assume  that
\[Z=(\omega\sigma X\backslash X)\backslash\bigcup_{n<\omega}S_n\]
with $S_n\in \mu_X ({\mathcal
T}^*_S(X))\cup\{\emptyset\}$ for all $n<\omega$, is another
representation of $Z$. Suppose that $\bigcup_{n<\omega}\phi
(Z_n)\neq\bigcup_{n<\omega}\phi (S_n)$. Without any loss of
generality we may assume that $\bigcup_{n<\omega}\phi
(Z_n)\backslash\bigcup_{n<\omega}\phi (S_n)\neq\emptyset$. Let $x\in
\bigcup_{n<\omega}\phi (Z_n)\backslash\bigcup_{n<\omega}\phi (S_n)$.
Let $ m<\omega$ be such that $x\in \phi(Z_m)$. Then since $x\notin
\bigcup_{n<\omega}\phi (S_n)$, there exists an $A\in {\mathcal
Z}(\omega\sigma Y\backslash Y)$ such that $x\in A$ and
$A\cap\bigcup_{n<\omega}\phi (S_n)=\emptyset$. Consider $A\cap\phi
(Z_m)\in \mu_Y ({\mathcal T}^*_S(Y))$. Let $B\in \mu_X ({\mathcal
T}^*_S(X))$ be such that $\phi(B)=A\cap\phi(Z_m)$.  Since $\phi
(B)\subseteq A$, we have $B\cap S_n=\emptyset$, for all $n<\omega$.
But  $B\subseteq Z_m\subseteq\bigcup_{n<\omega}
Z_n=\bigcup_{n<\omega} S_n$, which implies that $B=\emptyset$, which
is a contradiction. Therefore $\bigcup_{n<\omega}\phi
(Z_n)=\bigcup_{n<\omega}\phi (S_n)$, and $\psi$ is well defined.

To prove that $\psi$ is an order-isomorphism, let $S,Z\in{\mathcal
Z}(\omega\sigma X\backslash X)$ with $S\subseteq Z$. Assume that
$S\neq\emptyset$. We consider the following cases.

{\em Case 1)} Suppose that  $\Omega\notin Z$. Then
$\psi(S)=\phi(S)\subseteq\phi(Z)=\psi(Z)$.

{\em Case 2)} Suppose that $\Omega\notin S$ and $\Omega\in Z$.  Let
$Z=(\omega\sigma X\backslash X)\backslash\bigcup_{n<\omega}Z_n$,
with $Z_n\in \mu_X ({\mathcal T}^*_S(X))\cup\{\emptyset\}$, for all
$n<\omega$.  Then since $S\subseteq Z$, for each $n<\omega$, $S\cap
Z_n=\emptyset$, and therefore  $\phi (S)\cap \phi(Z_n)=\emptyset$.
We have
\[ \psi(S)=\phi(S)\subseteq(\omega\sigma Y\backslash
Y)\backslash\bigcup_{n<\omega}\phi (Z_n)=\psi(Z).\]

{\em Case 3)} Suppose that $\Omega\in S$. Let
\[ Z=(\omega\sigma X\backslash X)\backslash\bigcup_{n<\omega}Z_n
\mbox { and }S=(\omega\sigma X\backslash
X)\backslash\bigcup_{n<\omega}S_n\]
where for each $n<\omega$, the sets $S_n, Z_n\in \mu_X
({\mathcal T}^*_S(X))\cup\{\emptyset\}$. Since $S\subseteq Z$ we
have $\bigcup_{n<\omega}Z_n\subseteq\bigcup_{n<\omega}S_n$, and so
\[S=(\omega\sigma X\backslash
X)\backslash\bigcup_{n<\omega}(S_n\cup Z_n).\]
Therefore
\[\psi(S)=(\omega\sigma Y\backslash
Y)\backslash\bigcup_{n<\omega}\big(\phi (S_n)\cup\phi
(Z_n)\big)\subseteq(\omega\sigma Y\backslash
Y)\backslash\bigcup_{n<\omega}\phi (Z_n)=\psi(Z)\]
and thus $\psi$ is an order-homomorphism.

To show that $\psi$ is an order-isomorphism, we note that
$\phi^{-1}: \mu_Y ({\mathcal T}^*_S(Y))\rightarrow\mu_X({\mathcal
T}^*_S(X))$ is an order-isomorphism. Let
\[\gamma:{\mathcal Z}(\omega\sigma Y\backslash
Y)\rightarrow{\mathcal Z}(\omega\sigma X\backslash X)\]
be its induced order-homomorphism defined as above. Then
it is straightforward to see that $\gamma=\psi ^{-1}$, i.e., $\psi$
is an order-isomorphism, and thus ${\mathcal Z}(\omega\sigma
X\backslash X)$ and ${\mathcal Z}(\omega\sigma Y\backslash Y)$ are
order-isomorphic.  This implies that there exists a homeomorphism
$f: \omega\sigma X\backslash X\rightarrow\omega\sigma Y\backslash Y$
such that $f(Z)=\psi (Z)$, for every $Z\in{\mathcal Z}(\omega\sigma
X\backslash X)$. Therefore for any $M=\bigcup_{i\in L}X_i$ with
$L\subseteq I$ countable, since $M^*\in {\mathcal Z}(\omega\sigma
X\backslash X)$, we have $f(M^*)=\psi (M^*)=\phi(M^*)\subseteq\sigma
Y\backslash Y$. Therefore $f(\sigma X\backslash X)\subseteq\sigma
Y\backslash Y$ and thus $f(\Omega)=\Omega'$.  This shows that
$\sigma X\backslash X$ and $\sigma Y\backslash Y$ are homeomorphic.

{\em (2) implies (1).} If one of $X$ and $Y$,  say $X$, is
$\sigma$-compact, then since $\sigma Y\backslash Y$  is homeomorphic
to $X^*$, it is compact, and therefore as in Theorem \ref{PSDJ} ((2)
implies (1)) it follows that $Y$ also is $\sigma$-compact. Since by
1B of [13] $X^*\in{\mathcal Z}(\beta X)$, we have ${\mathcal
Z}(\beta X)\subseteq{\mathcal Z}(X^*)$, and thus from Theorem \ref{AADJ}
and Lemma \ref{UPPJ}, we have ${\mathcal T}^*_S(X)={\mathcal T}_C(X)$.
Similarly ${\mathcal T}^*_S(Y)={\mathcal T}_C(Y)$. Thus in this
case, the result follows from Theorem \ref{TDJ}.

The case when $X$ and $Y$ are both non-$\sigma$-compact, follows by
a slight modification of the proof we gave in Theorem \ref{PSDJ} ((2)
implies (1)).
\end{proof}

Our next result shows that the topology of $\sigma X\backslash X$
also  can be determined by the order structure of ${\mathcal
T}_{CL}(X)$. The following lemma follows from Theorem  2.3 and Lemma
\ref{ERWRPJ}.

\begin{lemma}\label{UKKPPJ}
For a  locally compact
paracompact
space $X$
\[\mu \big({\mathcal T}_{CL}(X)\big)=\big\{C\in {\mathcal Z}(X^*):C\supseteq \beta
X\backslash \sigma X\big\}\backslash\{\emptyset\}.\]
\end{lemma}

\begin{theorem}\label{ROEJ}
For  locally compact
paracompact spaces $X$ and $Y$ the following conditions  are
equivalent.
\begin{itemize}
\item[\rm(1)] ${\mathcal T}_{CL}(X)$ and  ${\mathcal T}_{CL}(Y)$ are order-isomorphic;
\item[\rm(2)] $\sigma X\backslash X$ and $\sigma Y\backslash Y$ are homeomorphic.
\end{itemize}
\end{theorem}

\begin{proof}
{\em (1) implies (2).} First suppose that one of $X$ and
$Y$, say $X$,  is $\sigma$-compact. Then since ${\mathcal T}(X)
={\mathcal T}_{L}(X)$, we have  ${\mathcal T}_{CL}(X)={\mathcal
T}_{C}(X)$. Suppose that $Y$ is not $\sigma$-compact.  As $|X^*|>1$,
there exists two disjoint non-empty zero-sets of $X^*$, which by
Theorem \ref{DDS} correspond to two elements of ${\mathcal T}_{C}(X)$ with
no common upper bound in ${\mathcal T}_{C}(X)$. But this is not
true, as we are assuming that ${\mathcal T}_{C}(X)$ and ${\mathcal
T}_{CL}(Y)$ are order-isomorphic, and by Lemma \ref{RAJ}0, any two
elements of ${\mathcal T}_{CL}(Y)$ have a common upper bound in
${\mathcal T}_{CL}(Y)$. The case $|X^*|\leq 1$ is not possible, as
$X$ is not pseudocompact, as it is paracompact and non-compact (see
Theorem \ref{YUGG}.20 of [4]). Therefore $Y$ is also $\sigma$-compact  and
${\mathcal T}_{CL}(Y)={\mathcal T}_{C}(Y)$, and thus by Theorem \ref{TDJ},
the spaces $\sigma X\backslash X=X^*$ and $\sigma Y\backslash Y=Y^*$
are homeomorphic.

Next suppose that $X$ and $Y$ are both non-$\sigma$-compact. Then by
condition (1) and the fact that $\mu_X$ and $\mu_Y$ are
order-anti-isomorphism, there exists an  order-isomorphism
$\phi:\mu_X({\mathcal T}_{CL}(X))\rightarrow\mu_Y({\mathcal
T}_{CL}(Y))$. Let $\omega\sigma X=\sigma X\cup\{\Omega\}$ and
$\omega\sigma Y=\sigma Y\cup\{\Omega'\}$ be one-point
compactifications. We define a function
\[\psi:{\mathcal Z}(\omega\sigma X\backslash
X)\rightarrow{\mathcal Z}(\omega\sigma Y\backslash Y)\]
and verify that it is an order-isomorphism.

Let
$X=\bigoplus_{i\in I}X_i$ and $Y=\bigoplus_{i\in J}Y_i$, with each
$X_i$ and $Y_i$ being a $\sigma$-compact non-compact subspace. Let
$Z\in{\mathcal Z}(\omega\sigma X\backslash X)$. Suppose that
$\Omega\in Z$. Then since $P=(\omega\sigma X\backslash X)\backslash
Z$ is a cozero-set in $\omega\sigma X\backslash X$, it is
$\sigma$-compact, and thus since $P\subseteq\sigma X\backslash X$,
we have $P\subseteq (\bigcup_{i\in K}X_i)^*$, for some countable
$K\subseteq I$. Now since $(\bigcup_{i\in K}X_i)^*$ is clopen in
$X^*$, we have $Q=(Z\backslash\{\Omega\})\cup(\beta
X\backslash\sigma X)\in{\mathcal Z}(X^*)$, and thus by Lemma \ref{RAJ}0,
we have  $Q\in \mu_X({\mathcal T}_{CL}(X))$. In this case we let
\[\psi(Z)=\big(\phi\big(\big(Z\backslash\{\Omega\}\big)\cup(\beta X\backslash\sigma
X)\big)\backslash(\beta Y\backslash\sigma Y)\big)\cup\{\Omega'\}.\]

Now suppose that  $\Omega\notin Z$. Then $ Z\subseteq \sigma
X\backslash X$ and therefore  $Z\subseteq (\bigcup_{i\in L} X_i)^*$,
for some countable $L\subseteq I$. Thus we have $Z=X^*\backslash
\bigcup_{n<\omega}Z_n$, where each $Z_n\in {\mathcal Z}(X^*) $
contains $\beta X\backslash\sigma X$. In this case we let
\[\psi(Z)=Y^*\backslash \bigcup_{n<\omega}\phi (Z_n).\]

We check that $\psi$ is well-defined. So suppose that
$Z=X^*\backslash \bigcup_{n<\omega}Z_n$ is a representation for
$Z\in{\mathcal Z}(\omega\sigma X\backslash X)$ with $\Omega\notin
Z$, such that each $Z_n\in {\mathcal Z}(X^*) $ contains $\beta
X\backslash\sigma X$. Since for each $n<\omega$, we have
$Y^*\backslash \phi(Z_n)\subseteq \sigma Y$, there exists a
countable $L\subseteq J$ such that for each $n<\omega$, we have
$Y^*\backslash \phi(Z_n)\subseteq(\bigcup_{i\in L}Y_i)^*$. Let $A$
be such that $\phi (A)=Y^*\backslash(\bigcup_{i\in L}Y_i)^*$. We
claim that
\[Y^*\backslash\bigcup_{n<\omega}\phi (Z_n)=\phi (A\cup
Z)\backslash \phi (A).\]

So suppose that $x\in Y^*\backslash\bigcup_{n<\omega}\phi (Z_n)$. If
$x\notin  \phi (A\cup Z)\backslash \phi (A)$,  then since $x\notin
\phi(Z_1)\supseteq \phi(A)$,  we have $x\notin \phi (A\cup Z)$, and
therefore there exists a $B\in {\mathcal Z}(Y^*)$ containing $x$
such that $B\cap \phi (A\cup Z)=\emptyset$, and $B\cap \phi
(Z_n)=\emptyset$, for each $n<\omega$.  Let $C$ be such that
$\phi(C)=B\cup\phi (A\cup Z)$, and for each $n<\omega$, let  $S_n$
be such that $\phi(S_n)=\phi(C)\cap \phi(Z_n)=\phi (A\cup Z)\cap
\phi(Z_n)$.
Then since for each $n<\omega$, we have $\phi(A)\subseteq \phi(Z_n)$
and $Z\cap
Z_n=\emptyset$, we have $(A\cup Z)\cap Z_n=A$. Clearly, by the way we
defined $S_n$, we have $S_n\subseteq
(A\cup Z)\cap Z_n=A$, and therefore  $\phi(S_n)\subseteq\phi(A)$.
But since $\phi(A)\subseteq\phi(Z_n)$,  we have
$\phi(A)\subseteq\phi(S_n)$, and thus for each $n<\omega$ we have
$\phi(C\cap Z_n)\subseteq \phi (C)\cap\phi
(Z_n)=\phi(S_n)=\phi(A)$.  Therefore $C\cap Z_n\subseteq A$, and thus
$C\backslash Z=C\cap
(\bigcup_{n<\omega}Z_n)\subseteq A$. Therefore $C\subseteq A\cup Z$,
and
we have $B\subseteq\phi (C)\subseteq\phi(A\cup Z)$, which is a
contradiction as $B\cap\phi(A\cup
Z)=\emptyset$. This shows that $Y^*\backslash
\bigcup_{n<\omega}\phi(Z_n)\subseteq\phi (A\cup Z)\backslash \phi
(A)$.

Now suppose that $x\in\phi (A\cup Z)\backslash \phi (A)$. Suppose
that  for some $n<\omega$, we have $x\in \phi(Z_n)$. Then $x\in
\phi(Z_n)\cap \phi(A\cup Z)=\phi(D)$, for some $D$. Clearly
$D\subseteq Z_n\cap( A\cup Z)\subseteq A$, and thus $x\in \phi (A)$,
which is contradiction. This proves our claim that
$Y^*\backslash\bigcup_{n<\omega} \phi(Z_n)=\phi (A\cup Z)\backslash
\phi (A)$.

Now suppose that
\[Z=X^*\backslash \bigcup_{n<\omega}S_n=X^*\backslash
\bigcup_{n<\omega}Z_n\]
are representations for $Z\in{\mathcal Z}(\omega\sigma
X\backslash X)$, with  $\Omega\notin Z$, such that each $S_n, Z_n\in
{\mathcal Z}(X^*) $ contains $\beta X\backslash\sigma X$. Choose a
countable $L\subseteq J$ such that
\[Y^*\backslash\phi (S_n)\subseteq\Big(\bigcup_{i\in L}Y_i\Big)^*\mbox {
and }Y^*\backslash\phi (Z_n)\subseteq\Big(\bigcup_{i\in L}Y_i\Big)^*\]
for each $ n<\omega$. Then by above we have
\[Y^*\backslash\bigcup_{n<\omega} \phi(S_n)=\phi (A\cup
Z)\backslash \phi (A)=Y^*\backslash\bigcup_{n<\omega} \phi(Z_n)\]
where $A$ is such that
$\phi(A)=Y^*\backslash(\bigcup_{i\in L}Y_i)^*$.

Next we show that $\psi$, as defined, is an order-isomorphism. So
suppose that  $S,Z\in{\mathcal Z}(\omega\sigma X\backslash X)$  with
$S\subseteq Z$. We consider the following cases.

{\em Case 1)} Suppose that  $\Omega\in S$. Then $\Omega\in Z$, and
clearly by the way we defined $\psi$, we have $\psi(S)\subseteq\psi(
Z)$.

{\em Case 2)} Suppose that $\Omega\notin S$ but $\Omega\in Z$. Let
$E=\phi((Z\backslash\{\Omega\})\cup(\beta X\backslash \sigma X))$
and let $S=X^*\backslash\bigcup_{n<\omega}S_n$, where for each
$n<\omega$, $S_n\in {\mathcal Z}(X^*)$ contains $\beta
X\backslash\sigma X$. Clearly $Y^*\backslash E\subseteq \sigma Y$.
Let the countable $L\subseteq J$ be such that for each $n<\omega$,
\[Y^*\backslash\phi (S_n)\subseteq\Big(\bigcup_{i\in L}Y_i\Big)^* \mbox {  and
} Y^*\backslash E\subseteq\Big(\bigcup_{i\in L}Y_i\Big)^*.\]
Then by above $\psi(S)=\phi (A\cup S)\backslash \phi
(A)$, where $\phi (A)=Y^*\backslash(\bigcup_{i\in L}Y_i)^*$. Since
$Y^*\backslash(\bigcup_{i\in L}Y_i)^*\subseteq E$, we have
$\phi(A)\subseteq E$, and therefore
$A\subseteq(Z\backslash\{\Omega\})\cup(\beta X\backslash\sigma X)$.
Now we have
\[\psi(S)\subseteq\phi\big(A\cup
S)\subseteq\phi\big(\big(Z\backslash\{\Omega\}\big)\cup(\beta X\backslash\sigma
X)\big)\]
and thus $\psi(S)\subseteq\psi(Z)$.

{\em Case 3)} Suppose that $\Omega\notin Z$, and let
\[S=X^*\backslash\bigcup_{n<\omega}S_n \mbox{ and
}Z=X^*\backslash\bigcup_{n<\omega}Z_n\]
where for each $n<\omega$, each of $S_n, Z_n\in {\mathcal
Z}(X^*)$ contain $\beta X\backslash\sigma X$. Since now
\[S=X^*\backslash\bigcup_{n<\omega}(S_n\cup Z_n)\]
we have
\[\psi(S)=Y^*\backslash\bigcup_{n<\omega}\phi(S_n\cup
Z_n)\subseteq Y^*\backslash\bigcup_{n<\omega}\phi(Z_n)=\psi(Z).\]

This shows that $\psi$ is an order-homomorphism. We note that
since
\[\phi^{-1}:\mu_Y\big({\mathcal T}_{CL}(Y)\big)\rightarrow\mu_X\big({\mathcal T}_{CL}(X)\big)\]
also is  an
order-isomorphism, if we denote by $\gamma:{\mathcal Z}(\omega\sigma
Y\backslash Y)\rightarrow{\mathcal Z}(\omega\sigma X\backslash X)$
its induced order-homomorphism as defined above, then it is easy to
see that $\gamma=\psi^{-1}$ and thus $\psi$ is an order-isomorphism.
Let $f:\omega\sigma X\backslash X\rightarrow\omega\sigma Y\backslash
Y$ be a homeomorphism such that $f(Z)=\psi(Z)$, for any $Z\in
{\mathcal Z}(\omega\sigma X\backslash X)$. Then since for each
countable $L\subseteq J$, we have
\[f\Big(\Big(\bigcup_{i\in L}X_i\Big)^*\Big)=\psi\Big(\Big(\bigcup_{i\in
L}X_i\Big)^*\Big)\subseteq\sigma Y\backslash Y\]
it follows that  $f(\sigma X\backslash X)=\sigma
Y\backslash Y$, and therefore $\sigma X\backslash X$ and $\sigma
Y\backslash Y$ are homeomorphic.

{\em (2) implies (1).} If one of $X$ and $Y$,  say $X$, is
$\sigma$-compact, then $\sigma Y\backslash Y$, being homeomorphic to
$X^* =\sigma X\backslash X$, is compact, and thus $Y$ is also
$\sigma$-compact. Thus by Theorem \ref{TDJ}, condition (1) holds.

Now suppose that both $X$ and $Y$ are non-$\sigma$-compact, and let
$f:\sigma X\backslash X\rightarrow\sigma Y\backslash Y$ be a
homeomorphism. We define an order-isomorphism $\phi:\mu_X({\mathcal
T}_{CL}(X))\rightarrow\mu_Y({\mathcal T}_{CL}(Y))$. Let
$D\in\mu_X({\mathcal T}_{CL}(X))$. By Lemma \ref{RAJ}0 we have $D\in
{\mathcal Z}(X^*)$ and $D\supseteq\beta X\backslash\sigma X$. Then
since $X^*\backslash D\subseteq \sigma X$ is $\sigma$-compact, using
the notations of Proposition \ref{RAWDJ}, there exists a countable
$L\subseteq I$ such that $X^*\backslash D\subseteq(\bigcup_{i\in
L}X_i)^*=A$. Now since $D\cap A\in {\mathcal Z}(A)$, we have
$f(D\cap A)\in {\mathcal Z}(f(A))$. But $A$ is open in $\sigma
X\backslash X$, and therefore $f(A)$ is open in $\sigma Y\backslash
Y$, and thus in $Y^*$, i.e., $f(A)$ is clopen in $Y^*$. Therefore
$B=f(D\cap A)\cup(Y^*\backslash f(A))\in {\mathcal Z}(Y^*)$. Let
$\phi(D)=f(D\cap(\sigma X\backslash X))\cup(\beta Y\backslash Y)$.
It is straightforward to check that $\phi(D)=G$, and thus $\phi$ is
well-defined. The function $\phi$ is clearly an order-homomorphism.
If we let $\psi:\mu_Y({\mathcal
T}_{CL}(Y))\rightarrow\mu_X({\mathcal T}_{CL}(X))$ be defined by
$\psi(D)=f^{-1}(D\cap(\sigma Y\backslash Y))\cup(\beta X\backslash
X)$, then $\psi=\phi^{-1}$, and therefore $\phi$ is  an
order-isomorphism.
\end{proof}

\section{The relation between various subsets of one-point extensions
of a locally compact space}

The order-anti-isomorphism $\mu$ enables us to obtain interesting
relations between the order structure of various sets of Tychonoff
extensions.

The following is a corollary of Theorems 2.3 and 2.8.

\begin{theorem}\label{YROEJ}
For any  locally compact space
$X$ we have
\[ {\mathcal T}_C^*(X)={\mathcal T}^*(X).\]
\end{theorem}

\begin{theorem}\label{YPPJ}
For any locally compact
$\sigma$-compact space $X$ we have
\[{\mathcal T}^*(X)={\mathcal T}_C(X).\]
\end{theorem}

\begin{proof}
Since $X$ is locally compact and $\sigma$-compact, by
1B of [13] we have  ${\mathcal Z}(X^*)\subseteq{\mathcal Z}(\beta
X)$. Now the result follows from Theorems 2.3 and
2.8.
\end{proof}

\begin{theorem}\label{YOPJ}
Let $X$ be a  locally compact
paracompact non-$\sigma$-compact space. Then
\[ {\mathcal T}_K(X)\cap {\mathcal T}_S(X)= {\mathcal T}^*_K(X).\]
\end{theorem}

\begin{proof}
Suppose that $Y=X\cup\{p\}\in {\mathcal T}_K(X)\cap
{\mathcal T}_S(X)$. Since $Y\in  {\mathcal T}_S(X)$, there exists a
closed neighborhood $U$ of $p$ in $Y$ such that $U\backslash\{p\}$
is $\sigma$-compact. Since $Y\in{\mathcal T}_K(X)$, there exists a
compact neighborhood $V$ of $p$ in $Y$. Then  $U\cap V$ is a
compact neighborhood of $p$ in $Y$ with $(U\cap V)\backslash\{p\}$
being $\sigma$-compact, and therefore by Lemma \ref{RRAJ} we have $Y\in
{\mathcal T}^*_K(X)$. By Lemma \ref{RRAJ}, we have ${\mathcal
T}^*_K(X)\subseteq{\mathcal T}_S(X)$, which completes the
proof.
\end{proof}

From Theorems 2.3 and \ref{EPSDJ} and Lemma \ref{WRORJ} we obtain the following
result.

\begin{theorem}\label{PYOPJ}
Let $X$ be a locally compact
paracompact  space. Then
\[{\mathcal T}_C(X)\cap{\mathcal T}_D(X)=\emptyset.\]
\end{theorem}

For a Tychonoff space $X$, if $S,T\in{\mathcal Z}(X)$, then $\mbox
{cl}_{\beta X}(S\cap T)= \mbox {cl}_{\beta X}S\cap \mbox {cl}_{\beta
X} T$. We use this fact  below.

\begin{lemma}\label{UKPJ}
Let $X$ be a locally compact
paracompact  space. If $Z\in {\mathcal Z}(\beta X)$ is such that
$Z\cap X=\emptyset $ then $\mbox {\em int}_{X^*}Z\subseteq\sigma X$.
\end{lemma}

\begin{proof}
Let $ Z\in {\mathcal Z}(\beta X)$ and  $Z\cap
X=\emptyset $. Suppose that $\mbox {int}_{X^*}Z\backslash\sigma
X\neq\emptyset$ and let $x\in \mbox {int}_{X^*}Z\backslash\sigma X$.
First using the same method as in Lemma 6.4 of [7] we find a
$T\in{\mathcal Z}(X)$  such that $x\in T^*\subseteq Z$. Since
$\{S^*:S\in{\mathcal Z}(X)\}$ is a base for closed subsets of $X^*$,
there exists an $S\in{\mathcal Z}(X)$ such that $x\in X^*\backslash
S^*\subseteq Z$. Now
\[S^*\cap\bigcap\big\{T^*:T\in{\mathcal Z}(X)\mbox { and } x\in
T^*\big\}=S^*\cap \{x\}=\emptyset\]
and therefore there exist
$T_1,\ldots, T_n \in{\mathcal Z}(X)$ such that $S^*\cap
T_1^*\cap\cdots\cap T_n^*=\emptyset$ and $ x\in T_i^*$, for $i=1,\ldots,
n$. Now if we let  $T=T_1\cap\cdots\cap T_n\in{\mathcal Z}(X)$, then
$x\in T^*=T_1^*\cap\cdots\cap T_n^*\subseteq X^*\backslash S^*\subseteq
Z$. Thus $\mbox {cl}_{\beta X}T\subseteq Z\cup X$. Let $Z=Z(f)$, for
some $f\in
C(\beta X, \mathbf{I})$. For each $n=1,2,\ldots$, we have
\[\mbox {cl}_{\beta X}T\backslash f^{-1}\big([0,1/n)\big)\subseteq
X=\bigoplus_{i\in
I}X_i\]
where each $X_i$ is a $\sigma$-compact non-compact
subspace. Therefore for each $n=1,2,\ldots$, there exists a finite set
$J_n\subseteq I$ such that
\[\mbox {cl}_{\beta X}T\backslash f^{-1}\big([0,1/n)\big)\subseteq
\bigcup_{i\in
J_n}X_i.\]
Let $J=J_1\cup J_2\cup\cdots$. Then
\[T\subseteq \mbox {cl}_{\beta X}T\backslash Z\subseteq
\bigcup_{n\geq 1}\big(\mbox {cl}_{\beta X}T\backslash
f^{-1}\big([0,1/n)\big)\big)\subseteq\bigcup_{i\in
J}X_i\]
and thus $\mbox {cl}_{\beta X}T\subseteq \mbox {cl}_{\beta
X}(\bigcup_{i\in
J}X_i)\subseteq\sigma X$. But this is a contradiction, as $x\in  \mbox
{cl}_{\beta X}T\backslash\sigma
X$. This proves the lemma.
\end{proof}

\begin{lemma}\label{UHKPJ}
Let $X$ be  a locally compact
space. Then every zero-set of $\beta X$ which misses $X$ is
regular-closed in $X^*$.
\end{lemma}

\begin{proof}
Let $Z\in{\mathcal Z}(\beta X)$ be such that $Z\cap
X=\emptyset$, and let $x\in Z$. If $x\notin \mbox {cl}_{X^*}\mbox
{int}_{X^*} Z$, then $x\in S$, for some $S\in{\mathcal Z}(\beta X)$
with $ S\cap \mbox {cl}_{X^*}\mbox {int}_{X^*} Z=\emptyset$. Let
$T=S\cap Z$. By Lemma 1\ref{YUGG}7 of [3], for a locally compact space $Y$,
any non-empty zero-set of $\beta Y$ which is contained in $Y^*$ has
non-empty interior in $Y^*$. Therefore  $\mbox {int}_{X^*}
T\neq\emptyset$. But this is a contradiction, as $\mbox {int}_{X^*}
T\subseteq \mbox {int}_{X^*} Z$ and $T\cap \mbox {int}_{X^*}
Z=\emptyset$. Therefore $x\in \mbox {cl}_{X^*}\mbox {int}_{X^*} Z$
and $Z$ is regular-closed in $X^*$.
\end{proof}

\begin{theorem}\label{FCG}
Let $X$ be a locally compact
paracompact  non-$\sigma$-compact space. Then ${\mathcal T}_L(X)$
contains an order-anti-isomorphic copy of ${\mathcal T}^*(X)$.
\end{theorem}

\begin{proof}
Suppose that $Z\in \mu({\mathcal T}^*(X))$. Then by
Theorem \ref{AADJ} we have $Z\in {\mathcal Z}(\beta X)$ and  $Z\cap
X=\emptyset $, and thus by Lemma \ref{UKPJ} we have $\mbox
{int}_{X^*}Z\subseteq \sigma X$. Therefore $ X^*\backslash \mbox
{int}_{X^*}Z\supseteq \beta X\backslash \sigma X$. Define a function
$\phi: \mu({\mathcal T}^*(X))\rightarrow \mu({\mathcal T}_L(X))$ by
$\phi (Z)=X^*\backslash \mbox {int}_{X^*}Z$. By Lemma \ref{ERWRPJ} the
function  $\phi$ is well-defined. Clearly for $S,T\in \mu({\mathcal
T}^*(X))$, if $S\subseteq T$, then $\phi(S)\subseteq \phi(T)$. The
converse also holds, as by Lemma \ref{UHKPJ} the sets  $S$ and $T$ are
regular closed in  $X^*$. Therefore $\phi$ and thus
$\psi=\mu^{-1}\phi\mu:{\mathcal T}^*(X)\rightarrow {\mathcal
T}_L(X)$ are order-anti-isomorphism onto their
images.
\end{proof}

\begin{theorem}\label{FTCG}
Let $X$ be a locally compact
paracompact  space. Then ${\mathcal T}_L(X)$ contains an
order-isomorphic copy of $ {\mathcal T}_C(X)$.
\end{theorem}

\begin{proof}
Let $\phi:\mu ({\mathcal T}_C(X))\rightarrow\mu
({\mathcal T}_L(X))$ be defined by $\phi (Z)= Z\cup (\beta
X\backslash \sigma X)$. By Theorem \ref{DDS} and Lemma \ref{ERWRPJ} the function
$\phi$ is well-defined.  If $\phi (Z_1)\subseteq \phi (Z_2)$, for
$Z_1, Z_2\in\mu ({\mathcal T}_C(X))$, then $Z_1\cap\sigma X\subseteq
Z_2\cap\sigma X $, and thus by Lemma \ref{WRPJ}, we have $Z_1\subseteq
Z_2$. Therefore if we let $\psi=\mu^{-1}\phi\mu: {\mathcal
T}_C(X)\rightarrow{\mathcal T}_L(X)$, then $\psi$ is an
order-isomorphism onto its image.
\end{proof}

From Theorems \ref{YROEJ} and \ref{FTCG} we obtain the following.

\begin{corollary}\label{FPP}
Let $X$ be a locally compact
paracompact space. Then ${\mathcal T}_{L}(X)$ contains an
order-isomorphic copy of ${\mathcal T}^*(X)$.
\end{corollary}

\begin{theorem}\label{GHCG}
Let $X$ be a  locally compact
paracompact space. Then ${\mathcal T}_{KL}(X)\backslash\{\omega X\}$
and  ${\mathcal T}_K^*(X)$ are order-anti-isomorphic.
\end{theorem}

\begin{proof}
Suppose that  $X$ is non-$\sigma$-compact and assume
the notations of Proposition \ref{RAWDJ}. By Theorem \ref{TODJ} and Lemma \ref{ERWRPJ} we
have
\[\mu\big({\mathcal T}_{KL}(X)\big)=\big\{C\in {\mathcal
B}(X^*):C\supseteq\beta X\backslash\sigma X\big\}.\]
Let $\phi:\mu({\mathcal
T}_{KL}(X))\backslash\{X^*\}\rightarrow\mu({\mathcal T}_K^*(X))$ be
defined by $\phi(C)=X^*\backslash C$. To see that $\phi$ is
well-defined, let $C\in {\mathcal B}(X^*)$ be such that
$C\supseteq\beta X\backslash\sigma X$. Then $X^*\backslash C$, being a
compact subset of $\sigma X$,
there exists a countable $J\subseteq I$ such that $X^*\backslash
C\subseteq
M^*$, where $M=\bigcup_{i\in J}X_i$.  Now $X^*\backslash C\in
{\mathcal Z}(M^*)$, and
therefore $X^*\backslash C\in {\mathcal Z}(\mbox {cl}_{\beta X}M)$ as
$M$ is $\sigma$-compact, and thus
  $X^*\backslash C\in {\mathcal Z}(\beta X)$, as $\mbox {cl}_{\beta
X}M$ is clopen in $\beta
  X$. Therefore by Lemma \ref{RAJ} we have $X^*\backslash C\in\mu({\mathcal
  T}_K^*(X))$. If $C\in\mu({\mathcal
  T}_K^*(X))$, then by Lemma \ref{RWRPJ} we have $C\subseteq \sigma X$, and
thus $X^*\backslash C\supseteq\beta X\backslash\sigma X$.
  Therefore $X^*\backslash C\in\mu({\mathcal
T}_{KL}(X))$.  This shows that $\phi$ is an order-anti-isomorphism
which proves the lemma in this case.

When $X$ is $\sigma$-compact ${\mathcal T}_{KL}(X)={\mathcal
T}_{K}(X)$, and by Lemma \ref{WRRAJ} we have ${\mathcal
T}_{K}^*(X)={\mathcal T}_{K}(X)$. Clearly in this case $\phi$ is
still a well-defined
order-anti-isomorphism.
\end{proof}

\begin{corollary}\label{FGGK}
For zero-dimensional locally
compact paracompact  spaces $X$ and $Y$ the following conditions are
equivalent.
\begin{itemize}
\item[\rm(1)] ${\mathcal T}_{KL}(X)$ and ${\mathcal T}_{KL}(Y)$ are order-isomorphic;
\item[\rm(2)] $\sigma X\backslash X$ and $\sigma Y\backslash Y$ are homeomorphic.
\end{itemize}
\end{corollary}

\begin{proof}
By the above lemma ${\mathcal T}_{KL}(X)$ and
${\mathcal T}_{KL}(Y)$ are order-isomorphic if and only if
${\mathcal T}_K^*(X)$ and ${\mathcal T}_K^*(Y)$ are
order-isomorphic. Now Theorem \ref{PSDJ} now completes the
proof.
\end{proof}

\begin{lemma}\label{UPJ}
Let $X$ be  a locally compact
paracompact space. If $Z\in {\mathcal Z}(X^*)$ contains $\beta
X\backslash\sigma X$, then $Z$ is regular-closed in $X^*$.
\end{lemma}

\begin{proof}
We assume that  $X$ is non-$\sigma$-compact. Suppose
that  $Z\in {\mathcal Z}(X^*)$ is such that $Z\supseteq\beta
X\backslash\sigma X$. Assume the notations of Proposition \ref{RAWDJ}. Since
$X^*\backslash Z\subseteq \sigma X$, and $X^*\backslash Z$ (being a
cozero-set  in $X^*$) is $\sigma$-compact, we have $X^*\backslash
Z\subseteq G^*$, where $G=\bigcup_{i\in J} X_i$ and $J\subseteq I$
is countable. Obviously since $X^*\backslash G^*\subseteq Z$ we have
\[(X^*\backslash G^*)\cup \mbox {cl}_{G^*}\mbox {int}_{G^*}(Z\cap
G^*)\subseteq \mbox {cl}_{X^*}\mbox {int}_{X^*}Z.\]

To show the
reverse inclusion suppose that $x\in \mbox {cl}_{X^*}\mbox
{int}_{X^*}Z$ and $x\in G^*$. Suppose that $x\notin \mbox
{cl}_{G^*}\mbox {int}_{G^*}(Z\cap G^*)$ and let $V$ be an open
neighborhood of $x$ in $G^*$ such that $V\cap \mbox
{int}_{G^*}(Z\cap G^*)=\emptyset$. But since $x\in \mbox
{cl}_{X^*}\mbox {int}_{X^*}Z$, we have $\emptyset\neq V\cap \mbox
{int}_{X^*}Z \subseteq  G^*\cap Z$, and thus $V\cap \mbox
{int}_{X^*}Z \subseteq V\cap \mbox {int}_{G^*}(Z\cap G^*)$, which is
a contradiction. Now since $G$ is $\sigma$-compact, it is
Lindel\"{o}f and therefore realcompact (see Theorem \ref{EPSDJ}.12 of [4]).
By Theorem 1\ref{YUGG}8 of [3], for a locally compact realcompact space
$T$, any zero-set of $T^*$ is regular-closed in $T^*$. Thus since
$G$ is also locally compact  $Z\cap G^*\in {\mathcal Z}(G^*)$ is
regular-closed in $G^*$. Therefore we have
\[\mbox {cl}_{X^*}\mbox {int}_{X^*}Z=(X^*\backslash G^*)\cup \mbox
{cl}_{G^*}\mbox {int}_{G^*}(Z\cap G^*)=(X^*\backslash G^*)\cup
(Z\cap G^*)=Z\]
which completes the
proof.
\end{proof}

\begin{theorem}\label{GSETG}
Let $X$ be  a locally compact
paracompact non-compact space. Then ${\mathcal T}_S(X)$ contains an
order-anti-isomorphic copy of  ${\mathcal
T}_{CL}(X)\backslash\{\omega X\}$.
\end{theorem}

\begin{proof}
Suppose that $X$ is non-$\sigma$-compact and let
$\phi:\mu({\mathcal
T}_{CL}(X))\backslash\{X^*\}\rightarrow\mu({\mathcal T}_{S}(X))$ be
defined by $\phi(Z)=X^*\backslash \mbox {int}_{X^*}Z$. To see that
$\phi$ is well-defined, we note that if $Z\in\mu({\mathcal
T}_{CL}(X))\backslash\{X^*\}$, then by Lemma \ref{RAJ}0  we have
$Z\supseteq \beta X\backslash\sigma X$, and thus since
$X^*\backslash Z\subseteq \sigma X$ is $\sigma$-compact, using the
notations of Proposition \ref{RAWDJ}, we have  $X^*\backslash Z\subseteq
G^*$, where $G=\bigcup_{i\in J} X_i$ and $J\subseteq I$ is
countable. Now since $X^*\backslash G^*\subseteq Z$, we have
$X^*\backslash G^*\subseteq \mbox {int}_{X^*}Z$, and thus
$\phi(Z)=X^*\backslash \mbox {int}_{X^*}Z\subseteq\sigma X$.
Therefore by Lemma \ref{EPPJ} we have
$\phi(Z)\in\mu({\mathcal T}_{S}(X))$. Now since by Lemma \ref{UPJ}
each $Z\in\mu({\mathcal T}_{CL}(X))$ is regular-closed in $X^*$, it
follows that $\phi$ and thus $\psi=\mu^{-1}\phi\mu:{\mathcal
T}_{CL}(X)\backslash\{\omega X\}\rightarrow{\mathcal T}_S(X)$ are
order-anti-isomorphisms onto their
images.
\end{proof}

We summarize some of the results of this section in the next
theorem. For this purpose we make the following notational
convention. For two partially ordered sets $P$ and $Q$ we write
$P\hookrightarrow Q$ ($P \mbox { (anti)}\hookrightarrow Q$,
respectively) if $Q$ contains an order-isomorphic
(order-anti-isomorphic, respectively) copy of $P$. We write $P\simeq
Q$ ($P \mbox { (anti)}\simeq Q$, respectively) if $P$ and $Q$ are
order-isomorphic (order-anti-isomorphic, respectively).

\begin{theorem}\label{OOTG}
Let $X$ be  a locally compact
paracompact  space. Then
\begin{itemize}
\item[\rm(1)] ${\mathcal T}^*(X) \mbox { \em  (anti)}\hookrightarrow {\mathcal
T}_L(X)$ (if $X$ is non-$\sigma$-compact);
\item[\rm(2)] ${\mathcal T}_C(X)\hookrightarrow {\mathcal T}_L(X)$;
\item[\rm(3)] ${\mathcal T}^*(X)\hookrightarrow {\mathcal T}_L(X)$;
\item[\rm(4)] ${\mathcal T}^*_K(X) \mbox { \em  (anti)}\simeq  {\mathcal T}_{KL}(X)\backslash\{\omega X\}$;
\item[\rm(5)] ${\mathcal T}_{CL}(X)\backslash\{\omega X\}\mbox{\em(anti)}\hookrightarrow {\mathcal T}_S(X)$.
\end{itemize}
\end{theorem}

\begin{question}\label{UUTG}
In  Theorems \ref{FCG}, \ref{FTCG} and
\ref{GSETG}, which one-point extensions constitute exactly the image of
$\psi$?
\end{question}

\section{The existence of minimal and maximal elements in various sets
of one-point extensions}

We start this section  with  the following simple observation.

\begin{theorem}\label{YUGG}
Let $X$ be a locally compact
non-compact space. Then the maximal elements  of ${\mathcal T}(X)$
are exactly those of the form $X\cup\{p\}\subseteq \beta X$, for
$p\in X^*$. Moreover,  ${\mathcal T}(X)$ has a minimum, namely, its
one-point compactification.
\end{theorem}

\begin{theorem}\label{YOOG}
Let $X$ be a locally compact
non-compact space. Then
\begin{itemize}
\item[\rm(1)] ${\mathcal T}^*(X)$ has no maximal element.
\item[\rm(2)] The following conditions are equivalent.
\begin{itemize}
\item[\rm(a)] ${\mathcal T}^*(X)$ has a minimal element;
\item[\rm(b)] ${\mathcal T}^*(X)$ has a minimum;
\item[\rm(c)] $\upsilon X$ is locally compact and $\sigma$-compact;
\item[\rm(d)] {\em (Hager; cited in  [6], Theorem 2.9)} $X=\bigcup_{n<\omega}A_n$, where for each $n<\omega$, $A_n$ is pseudocompact and $A_n$ and $X\backslash A_{n+1}$ are completely separated in $X$.
\end{itemize}
\end{itemize}
\end{theorem}

\begin{proof}
1) Suppose to the contrary that $Y$ is a maximal
element of ${\mathcal T}^*(X)$ and let $S=\mu (Y)$. By Theorem \ref{AADJ},
we have $S\in{\mathcal Z}(\beta X)$ and $S\cap X=\emptyset$. Clearly
$|S|=1$, for otherwise, there is a non-empty zero-set of $\beta X$
properly contained in $S$, which contradicts the maximality of $Y$.
Let $T=\beta X\backslash S$. By Theorem 1\ref{YUGG}5 of [3], for any
$\sigma$-compact non-compact space $G$, we have $|G^*|\geq
2^{2^{\aleph_0}}$. Therefore since $T$ is $\sigma$-compact
non-compact,  we have $| \beta T\backslash T|\geq 2^{2^{\aleph_0}}$.
But this is clearly a contradiction, as $\beta T\backslash T=\beta
X\backslash (\beta X\backslash S)=S$. Therefore ${\mathcal T}^*(X)$
has no maximal element.

2) The equivalence of conditions (a) and (b) follows from the fact
that by Theorem \ref{AADJ}, for any $Y_1,Y_2\in {\mathcal T}^*(X)$ we have
$Y_1\wedge Y_2\in {\mathcal T}^*(X)$.

To show that condition (b) implies  (c), suppose that ${\mathcal
T}^*(X)$ has a minimum element $Y$. Let $C=\mu(Y)$. Then since by
Theorem \ref{AADJ} every non-empty zero-set of $\beta X$ which is disjoint
from $X$ corresponds to an element of ${\mathcal T}^*(X)$, it is
contained in $C$, and therefore since $\upsilon X$ is the
intersection of all cozero-sets of $\beta X$ which contain $X$, we
have $\beta X\backslash C\subseteq\upsilon X$. Clearly $\upsilon
X\subseteq \beta X\backslash C$, and therefore $\upsilon X=\beta
X\backslash C$ being a cozero-set in $\beta X$ is $\sigma$-compact.
It is also locally compact as it is open in $\beta X$. Thus
condition (c) holds.

Now suppose that condition (c) holds. Then since $\upsilon X$ is
locally compact and $\sigma$-compact, by 1B of [13], we have $\beta
X\backslash \upsilon X\in {\mathcal Z}(\beta X)$. We assume that $X$
is not pseudocompact, as otherwise by Corollary \ref{AWDJ} we have
${\mathcal T}^*(X)=\emptyset$. Let $Y\in{\mathcal T}^*(X)$ be such
that $\mu (Y)=\beta X\backslash \upsilon X$. Then clearly for every
$S\in{\mathcal T}^*(X)$, we have  $\upsilon X\subseteq \beta
X\backslash \mu (S)$ and thus $\mu(S)\subseteq\mu (Y)$, i.e., $Y\leq
S$, which shows that ${\mathcal T}^*(X)$ has a
minimum.
\end{proof}

A space is called {\em almost realcompact} if it is the perfect
(continuous) image of a realcompact space (see [12, 6U]).

\begin{corollary}\label{TYUGG}
Let $X$ be a locally compact
non-compact space. Consider the following conditions.
\begin{itemize}
\item[\rm(1)] $X$ is a $P$-space;
\item[\rm(2)] $X$ is almost realcompact;
\item[\rm(3)] $X$ is weakly paracompact;
\item[\rm(4)] $X$ is Dieudonn\'{e}-complete;
\item[\rm(5)] {\em [MA+$\neg$CH]} $X$ is perfectly normal.
\end{itemize}
Assume that $X$ satisfies one of the above conditions.
Then ${\mathcal T}^*(X)$ has a minimum if and only if $X$ is
$\sigma$-compact.
\end{corollary}

\begin{proof}
Suppose that ${\mathcal T}^*(X)$ has a minimum.
First assume that one of conditions (1)-(3) and (5) holds. Then by
6AB of [12] the set  $\beta X\backslash\upsilon X$ is dense in
$X^*$. But by Theorem \ref{YOOG} $\upsilon X$ is locally compact, and thus
$\beta X\backslash\upsilon X$ is  closed in $\beta X$. Therefore
$\beta X\backslash\upsilon X=X^*$, and thus $X=\upsilon X$, which by
Theorem \ref{YOOG} is $\sigma$-compact.

Suppose that condition (4) holds. Then since ${\mathcal T}^*(X)$ has
a minimum, by Theorem \ref{YOOG} we have $X=\bigcup_{n<\omega}A_n$, where
for each $n<\omega$, $A_n$ is pseudocompact.  Since
Dieudonn\'{e}-completeness is closed hereditary, each $\mbox
{cl}_XA_n$ is Dieudonn\'{e}-complete. But pseudocompactness and
compactness coincide in the realm of Dieudonn\'{e}-complete spaces
(see 8.\ref{YUGG}3 of [4]) therefore each  $\mbox {cl}_XA_n$ being
pseudocompact is compact, and $X=\bigcup_{n<\omega}\mbox {cl}_XA_n$
is $\sigma$-compact.

The converse is clear, as if $X$ is $\sigma$-compact, then $\omega
X$ is the minimum of ${\mathcal
T}^*(X)$.
\end{proof}

It is worth to note that $X=\omega$ is the only locally compact
non-compact $P$-space for which ${\mathcal T}^*(X)$ has a minimum.
This is because if for a locally compact non-compact $P$-space $X$,
${\mathcal T}^*(X)$ has a minimum, then by Theorem \ref{YOOG} we have
$X=\bigcup_{n<\omega}A_n$, where each $A_n$ is pseudocompact,
and since each $A_n$ is also a $P$-space, it is finite. Therefore
$X$ is a countable $P$-space, and thus it is discrete (see 4K of
[5]).

\begin{theorem}\label{SYOOG}
Let $X$ be a locally compact
non-compact space. Then ${\mathcal T}_C(X)$  has  a minimum. If $X$
is realcompact or paracompact then ${\mathcal T}_C(X)$ has no
maximal element.
\end{theorem}

\begin{proof}
Since $\omega X\in{\mathcal T}_C(X)$, it is clear that
${\mathcal T}_C(X)$ has a minimum.

Now suppose that $X$ is realcompact. Suppose that  ${\mathcal
T}_C(X)$ has a maximal element $Y$. If $G=\mu(Y)$, then $|G|=1$. As
otherwise, $G$ properly contains a non-empty zero-set of $X^*$,
contradicting the maximality of $Y$.  Let $G=\{p\}$ and let $S\in
{\mathcal Z}(\beta X)$  be such that $p\in S$ and $ S\cap
X=\emptyset$. Let $T\in {\mathcal Z}(\beta X)$  be such that
$G=T\backslash X$. Then $G=T\cap S\in {\mathcal Z}(\beta X)$. Now
$\beta X\backslash G$ is almost compact and thus pseudocompact (see
6J of [5]). But it is also $\sigma$-compact as it is a cozero-set in
$\beta X$, therefore, it is compact. This contradictions shows that
in this case ${\mathcal T}_C(X)$ has no maximal element.

Next suppose that $X$ is paracompact. We
may assume that $X$ is not $\sigma$-compact, as $\sigma$-compact
spaces are realcompact. Suppose that  ${\mathcal T}_C(X)$ has a
maximal element $Y$ and let $H=\mu (Y)$. As above $H=\{p\}$, for
some $p\in X^*$. Since by Lemma \ref{WRORJ} $H\cap \sigma X\neq \emptyset$,
we have $p\in \sigma X$. Assume the notations of Proposition \ref{RAWDJ} and
let $J\subseteq I$ be  countable and  such that $p\in \mbox
{cl}_{\beta X}M$, where $M=\bigcup _{i\in J}X_i$. Since
$H\in{\mathcal Z}(X^*)$, we have $H\in{\mathcal Z}(M^*)$. Let
$S\in{\mathcal Z}(\mbox {cl}_{\beta X}M)$ be such that $H=S\cap
M^*$. Now since $M$ is $\sigma$-compact, $M^*\in{\mathcal Z}(\mbox
{cl}_{\beta X}M)$, and thus $H\in{\mathcal Z}(\mbox {cl}_{\beta
X}M)$. But $M^*$ is itself clopen in $\beta X$ and therefore $H\in
{\mathcal Z}(\beta X)$, which as in the above part we get a
contradiction. Therefore ${\mathcal T}_C(X)$ has no maximal
element.
\end{proof}

In connection with the above theorem we remark that, assuming that
every cardinal number is non-measurable, paracompact spaces are
realcompact (see Corollary \ref{HGH}(m) of [12]).

\begin{theorem}\label{UOOG}
Let $X$ be a locally compact
non-compact space. Then ${\mathcal T}_K(X)$  has a minimum.
${\mathcal T}_K(X)$ may or may not have maximal elements.
\end{theorem}

\begin{proof}
It is clear that  ${\mathcal T}_K(X)$ has a minimum,
namely, its one-point compactification.

Let $X=\bigoplus_{i\in I}
X_i$, where $I\neq \emptyset$ and for each $i\in I$, $X_i=[0,1)$.
Since for each $i\in I$, we have $X_i^*\in {\mathcal B}(X^*)$, there
exists a $Y_i\in {\mathcal T}_K(X)$ such that $\mu(Y_i)=X_i^*$. Now
since each $X_i^*$ does not properly contains any non-empty element
of ${\mathcal B}(X^*)$, the corresponding $Y_i$'s are maximal
elements of ${\mathcal T}_K(X)$.

Now let $X$ be an uncountable discrete space and let $C\in{\mathcal
B}(X^*)\backslash\{\emptyset\}$. By Lemma \ref{WRORJ}, we have $C\cap \sigma
X\neq\emptyset$. Let $A$ be a countable subset of $X$ such that
$C\cap A^*\neq\emptyset$. Now $C\cap A^*$ is clopen in $A^*\simeq
\omega^*$ and therefore it properly contains a non-empty clopen
subset of $A^*$, which is therefore a clopen subset of $X^*$. By
Theorem \ref{TODJ} this shows that ${\mathcal T}_K(X)$ has no maximal
element.
\end{proof}

\begin{lemma}\label{YHG}
Let $X$ be a normal space. Then
every one-point regular extension of $X$ also is normal.
\end{lemma}

\begin{proof}
Suppose that  $Y=X\cup\{p\}$ is a one-point regular
extension of $X$. Let $A$ and $B$ be disjoint closed subsets of $Y$.
If $A$ and $B$ are closed subsets of $X$, then obviously they can be
separated by disjoint open sets in $X$, and thus in  $Y$. So suppose
that  $p\in A$, and let $U$ and $V$ be disjoint open subsets of $X$
such that $ A\cap X\subseteq U$ and $B\subseteq V$. Let $W$ be an
open neighborhood of $p$ in $Y$ such that $B\cap\mbox
{cl}_YW=\emptyset$. Then $U\cup W$ and $V\backslash \mbox {cl}_YW$
are disjoint open subsets of $Y$ which separate $A$ and $B$,
respectively.
\end{proof}

We call a  space $X$ {\em locally Lindel\"{o}f}, if every $x\in X$
has an open neighborhood $U$ in $X$ such that $\mbox {cl}_XU$ is
Lindel\"{o}f.

\begin{theorem}\label{PPG}
Let $X$ be a paracompact
non-Lindel\"{o}f space. Then
\begin{itemize}
\item[\rm(1)] ${\mathcal T}_D(X)$  has a minimum if and only if $X$ is locally Lindel\"{o}f;
\item[\rm(2)] If $X$ is moreover locally compact, then ${\mathcal T}_D(X)$ has a maximal element.
\end{itemize}
\end{theorem}

\begin{proof}
1) Suppose  that $X$ is locally
Lindel\"{o}f. Let $\delta X= X\cup\{\Delta\}$, where $\Delta\notin
X$. Define a topology on $\delta X$ consisting of open sets of $X$
together with sets of the form $\{\Delta\}\cup (X\backslash F)$,
where $F$ is a closed Lindel\"{o}f subspace of $X$. It is
straightforward to see that $\delta X$ is a topological space which
contains $X$ as a dense subspace. We first check that $X$ is
Hausdorff.

So suppose that  $a,b\in \delta X$ and $a\neq b$. If $a,b\in X$,
then clearly they can be separated by disjoint open sets in $X$, and
thus in $\delta X$. Suppose that $a=\Delta$ and let $U$ be an open
neighborhood of $b$ in $X$ such that $\mbox {cl}_XU$ is
Lindel\"{o}f. Then the sets $\{\Delta\}\cup (X\backslash \mbox
{cl}_XU)$ and $U$ are disjoint open sets of $\delta X$ separating
$a$ and $b$, respectively.

Next we show that $\delta X$ is regular.
So suppose that $y\in \delta X$ and let $U$ be an open neighborhood
of $y$ in $\delta X$. First suppose that $y=\Delta$. Then $U$ is of
the form $\{\Delta\}\cup (X\backslash F)$, for some closed
Lindel\"{o}f subspace $F$ of $X$. For each $x\in F$, let $U_x$ be an
open neighborhood of $x$ in $X$ with $\mbox {cl}_XU_x$ being
Lindel\"{o}f. Since $F$ is Lindel\"{o}f, there exist $x_1,
x_2,\ldots\in F$ such that $F\subseteq \bigcup_{n\geq 1}U_{x_n}$.
Consider the open cover ${\mathcal U}= \{U_{x_n}\}_{n\geq
1}\cup\{X\backslash F\}$ of $X$. Then since $X$ is paracompact,
there exists a locally finite open refinement ${\mathcal V}$ of
${\mathcal U}$. Let
\[G=\mbox {cl}_X\Big(\bigcup\{V\in {\mathcal V}: V\cap
F\neq\emptyset\}\Big).\]
Then since if $V\in {\mathcal V}$ and $V\cap
F\neq\emptyset $, then $V\subseteq U_{x_n}$ for some $n\geq 1$, and
$ {\mathcal V}$ is locally finite, we have
\[G=\bigcup\{\mbox {cl}_X V: V\in {\mathcal V} \mbox { and } V\cap
F\neq\emptyset\}\subseteq \bigcup_{n\geq 1} \mbox {cl}_X
U_{x_n}=H.\]
Thus $G$ being a closed subset of the Lindel\"{o}f space
$H$ is itself Lindel\"{o}f. Now we note that
\[F\subseteq \bigcup\{V\in {\mathcal V}: V\cap
F\neq\emptyset\}\subseteq \mbox {int}_X G\]
and therefore we have
\[\mbox {cl}_{\delta X}\big(\{\Delta\}\cup (X\backslash
G)\big)=\{\Delta\}\cup \mbox {cl}_X (X\backslash G)\subseteq
\{\Delta\}\cup (X\backslash F)\]
i.e., $\{\Delta\}\cup (X\backslash G)$ is an open
neighborhood of $y$ in $\delta X$ whose closure in $\delta X$ is
contained in $U$. Now suppose that $y\in X$ and let $V$ and $W$ be
open neighborhoods of $y$ in $X$ with $\mbox {cl}_XV$ being
Lindel\"{o}f and $\mbox {cl}_XW\subseteq U\cap V$. Then $\mbox
{cl}_{\delta X}W=\mbox {cl}_XW\subseteq U$. This shows that $\delta
X$ is regular, and since it is Lindel\"{o}f, it is normal.

Clearly $\Delta\notin\mbox
{cl}_{\delta X}F$, for any closed Lindel\"{o}f subset $F$ of $X$,
and thus $\delta X\in {\mathcal T}_D(X)$.  To show that  $\delta X$
is a minimum, suppose that $Y=X\cup\{p\}\in{\mathcal T}_D(X)$ and
let $f:Y\rightarrow\delta X$ be defined such that  $f|X=\mbox
{id}_X$ and $f(p)=\Delta$. Then since any open neighborhood of
$\Delta$ in $\delta X$ is of the form $V=\{\Delta\}\cup (X\backslash
F)$, for some closed Lindel\"{o}f subset $F$ of $X$, and $p\notin
\mbox {cl}_YF$, there exists an open neighborhood $U$ of $p$ in $Y$
such that $U\cap F=\emptyset$, and therefore  $f(U)\subseteq V$,
i.e., $f$ is continuous at $p$ and thus on $Y$. This shows that
$Y\geq \delta X$, which completes the proof of this part.

Next suppose that ${\mathcal T}_D(X)$  has a minimum, say
$Y=X\cup\{p\}$.
Suppose that  $X$ is not locally
Lindel\"{o}f and let $U$ be an open subset of $X$ such that $p\notin
\mbox {cl}_YU$ and $\mbox {cl}_XU$ is not Lindel\"{o}f. Let
$\{U_i\}_{i\in I}$ be a
cover of $\mbox {cl}_XU$ consisting of open subsets of $X$ with no
countable subcover. Refining  $\{U_i\}_{i\in I}$ by using
regularity, we may assume
that $\mbox {cl}_XU$ is not covered by any countable union of closures
of $U_i$'s in $X$. Let ${\mathcal V}$ be a locally finite open
refinement of $\{U_i\}_{i\in I}\cup\{X\backslash \mbox {cl}_XU\}$.
Let
\[{\mathcal W}=\{V\in{\mathcal V}: V\cap U\neq\emptyset
\}=\{W_j\}_{j\in J}\]
which is faithfully indexed. It is clear that $J$ is
uncountable, as otherwise, since $\{W_j\}_{j\in J}$ covers $U$ and
they are locally finite $\mbox {cl}_XU\subseteq \bigcup_{i\in
J}\mbox {cl}_XW_j$, which is a contradiction, as each $W_j$ is a
subset of some $U_i$. For each $j\in J$, let $x_j\in W_j\cap U$. Let
$A=X\cup \{q\}$, with $q\notin X$, and define a topology on $A$
consisting of open sets of $X$ together with sets of the form
$B\cup\{q\}$, where $B\subseteq X$, the set $B\cup\{p\}$ is open in
$Y$, and $B\supseteq \bigcup_{j\in J\backslash L} C_j$, where
$L\subseteq J$ is countable,  and for each $j\in J\backslash L$,
the set  $C_j$ is an open neighborhood of $x_j$ in $X$ contained in
$W_j\cap U$. Then it is easy to verify that $A$ is a topological
space containing $X$ as a dense subspace.

To see that
$A$ is a $T_1$-space, let $x\in X$. Since ${\mathcal W}$ is locally
finite, there exists a finite set $L\subseteq J$ such that $x\notin
W_j$, for any $j\in J\backslash L$.
Let
\[B=(D\cap X)\cup\bigcup_{j\in J\backslash L}(W_j\cap U)\]
where $D$ is
an open neighborhood of $p$ in $Y$ not containing $x$. Then
$B\cup\{q\}$ is an open neighborhood of $q$ in $A$  which does not
contain $x$.

Next we show that $A$ is regular. So suppose that $y\in A$ and let
$W$ be  an open neighborhood of  $y$ in $A$. First suppose that
$y=q$. Then $W=B\cup\{q\}$, where $B\subseteq X$, the set
$B\cup\{p\}$ is open in $Y$ and $B\supseteq \bigcup_{j\in
J\backslash L} C_j$, for some countable $L\subseteq J$ and open sets
$C_j$'s of $X$ such that $x_j\in C_j\subseteq W_j\cap U$. Let $G$ be
an open neighborhood of $p$ in $Y$ such that $\mbox {cl}_YG\subseteq
B\cup\{p\}$, and for each $j\in J\backslash L$, let  $H_j$ be an
open neighborhood of $x_j$ in $X$ with $\mbox {cl}_X H_j\subseteq
C_j$. Then
\[V=(G\cap X)\cup\bigcup_{j\in J\backslash L}H_j\cup \{q\}\]
is an open neighborhood of $q$ in $A$ and
\[ \mbox {cl}_AV=\mbox {cl}_X(G\cap X)\cup\bigcup_{j\in J\backslash
L}\mbox {cl}_XH_j\cup \{q\}\subseteq B\cup\bigcup_{j\in J\backslash
L}C_j\cup \{q\}=B\cup\{q\}=W. \]
Now suppose that  $y\in X$. Let $F$ and $G$ be disjoint
open neighborhoods of $p$ and $y$ in $Y$, respectively. Let $H$ be
an open neighborhood of $y$ in $X$ such that $\mbox {cl}_X
H\subseteq W\cap G$, and let $K$ be an open neighborhood of $y$ in
$X$ intersecting at most finitely many of $W_j$'s.  Let the finite
set $L\subseteq J$ be such that $K\cap W_j=\emptyset$, for any $j\in
J\backslash L$. Let
\[D=(F\cap X)\cup\bigcup _{j\in J\backslash L}W_j.\]
Then $D\cup \{q\}$ is an open neighborhood of $q$ in $A$
missing $K\cap H$. Therefore $q\notin \mbox {cl}_A (K\cap H)$, and
thus $K\cap H$  is an open neighborhood of $y$ in $A$ such that
\[\mbox {cl}_A(K\cap H)=\mbox {cl}_X(K\cap H)\subseteq \mbox
{cl}_X H \subseteq W.\]
This shows that $A$ is regular and thus by Lemma
\ref{YHG}, it is also normal.

Now let $P$ be a closed Lindel\"{o}f subspace of
$X$. For each $x\in P\cap \mbox {cl}_XU$, let  $V_x$ be an open
neighborhood of $x$ in $X$ which intersects only finitely many of
$W_j$'s, say for $j\in L_x$, where $L_x\subseteq J$ is finite. Since
$P\cap \mbox {cl}_XU$ is closed in $P$, it is Lindel\"{o}f, and
therefore since
\[P\cap \mbox {cl}_XU\subseteq \bigcup\{V_x:x\in P\cap
\mbox{cl}_XU\}\]
there exist $x_1, x_2,\ldots\in P\cap \mbox {cl}_XU$ such
that $P\cap \mbox {cl}_XU\subseteq\bigcup_{n\geq 1}V_{x_n} $. Let
$L=\bigcup_{n\geq 1}L_{x_n}$. Then clearly for each $j\in
J\backslash L$, we have $W_j\cap P\cap \mbox {cl}_XU=\emptyset$. Now
since $Y=X\cup\{p\}\in {\mathcal T}_D(X)$, we have $p\notin \mbox
{cl}_YP$, and thus there exists an open neighborhood $M$ of $p$ in
$Y$ such that $M\cap P=\emptyset $. Let
\[B=(M\cap X)\cup\bigcup_{j\in J\backslash L}(W_j\cap U).\]
Then $B\cup\{q\}$ is an open neighborhood of $q$ in $A$,
and we have
\[P\cap\big(B\cup\{q\}\big)=P\cap\bigcup_{j\in J\backslash L}(W_j\cap U)
=\emptyset.\]
This shows that $ q\notin \mbox {cl}_A P$, and thus $
A\in {\mathcal T}_D(X)$. But this is impossible, as by the way we
defined neighborhoods of $q$ in $A$, each of them contains an $x_j$,
for some $j\in J$, and therefore has non-empty intersection with
$U$, which contradicts  the fact that $A\geq Y$. This shows that $X$
is locally Lindel\"{o}f.

2) This is clear as in this case by Theorem \ref{EPSDJ}, any
$Y=X\cup\{p\}$, where $p\in \beta X\backslash\sigma X$, belongs to
${\mathcal T}_D(X)$ and it is obviously
maximal.
\end{proof}

We note in passing that a hedgehog with an uncountable number of
spines is an example of a paracompact space which is not locally
Lindel\"{o}f.

\begin{theorem}\label{GHHGG}
Let $X$ be a locally compact
paracompact non-$\sigma$-compact space. Then ${\mathcal T}_L(X)$ has
both maximal and minimal elements.
\end{theorem}

\begin{proof}
This follows from Lemma \ref{ERWRPJ} and the fact that both
$\beta X\backslash \sigma X$ and $X^*$ belong to $\mu({\mathcal
T}_L(X))$.
\end{proof}

\begin{theorem}\label{HJL}
Let $X$ be a locally compact
paracompact non-$\sigma$-compact space. Then ${\mathcal T}_S(X)$ has
a  maximal element but does not have a minimal element.
\end{theorem}

\begin{proof}
Clearly every element of the form $Y=X\cup\{p\}$, for
$p\in \sigma X\backslash X$, is a maximal element of ${\mathcal
T}_S(X)$.

Suppose that $Y\in {\mathcal T}_S(X)$. Assume the notations of
Proposition \ref{RAWDJ}. Then since by Lemma \ref{EPPJ} $\mu (Y)\subseteq \sigma
X$, we have $\mu(Y)\subseteq (\bigcup_{i\in J}X_i)^*$, for some
countable $J\subseteq I$. Let the countable $L\subseteq I$ properly
contain $J$. Then if $T\in {\mathcal T}_S(X)$ is such that
$\mu(T)=(\bigcup_{i\in L}X_i)^*$, we have $T<Y$. Therefore
${\mathcal T}_S(X)$ has no  minimal
element.
\end{proof}

\begin{theorem}\label{HJL}
Let $X$ be a locally compact
non-pseudocompact space. Then ${\mathcal T}_P(X)$ has both minimum
and maximum.
\end{theorem}

\begin{proof}
It is clear that $\omega X$ is the minimum of
${\mathcal T}_P(X)$. Let $C=\beta X\backslash\mbox{int}_{\beta
X}\upsilon X$. Then since $X$ is locally compact $X\subseteq
\mbox{int}_{\beta X}\upsilon X$, and thus $C\subseteq X^*$. Since
$X$ is not pseudocompact $C\neq \emptyset$. By Theorem \ref{SDJ} there
exists a $Y\in{\mathcal T}_P(X)$ such that $\mu(Y)=C$. If $S\in
{\mathcal T}_P(X)$, then since by Theorem \ref{SDJ} we have
$\mu(S)\supseteq\beta X\backslash\upsilon X$, it follows that $\beta
X\backslash\mu(S)\subseteq\upsilon X$, and therefore $\beta
X\backslash\mu(S)\subseteq\mbox{int}_{\beta X}\upsilon X$ . Thus
$\mu(Y)\subseteq\mu(S)$, and therefore $S\leq Y$. This shows that
$Y$ is maximum in ${\mathcal
T}_P(X)$.
\end{proof}

\begin{theorem}\label{HGH}
Let $X$ be a locally compact
non-compact space. Then the minimum of ${\mathcal T}^*(X)$, if
exists, is the unique pseudocompact element of ${\mathcal T}^*(X)$
{\em (compare  with Theorem \ref{UOOG} and Corollary \ref{YHG} of  [B])}.
\end{theorem}

\begin{proof}
Suppose that $Y$ is the minimum of ${\mathcal T}^*(X)$
and let $C=\mu(Y)$. By the proof of Theorem \ref{YOOG} ((b) implies (c)) we
know that $C=\beta X\backslash\upsilon X$. Therefore by Theorem \ref{SDJ}, the space
$Y$ is pseudocompact.

If $S$ is another pseudocompact element of
${\mathcal T}^*(X)$, then by Theorem \ref{SDJ} we have
$\mu(S)\supseteq\beta X\backslash\upsilon X$. On the other hand, by
Theorem \ref{AADJ}, $\mu(S)$ is a zero-set in $\beta X$ contained in $X^*$,
which implies that $\mu(S)\subseteq\beta X\backslash\upsilon X$.
Thus $\mu(S)=\beta X\backslash\upsilon X=\mu(Y)$, and therefore
$S=Y$. This shows the uniqueness of
$Y$.
\end{proof}

The following result partially answers Question \ref{UPSDJ}.

\begin{theorem}\label{TGHG}
Let $X$ and $Y$ be locally
compact non-compact spaces such that $X=\bigcup_{n<\omega}A_n$ and
$Y=\bigcup_{n<\omega}B_n$, where each $A_n$ and $B_n$ is
pseudocompact and for each $n<\omega$ the pairs $A_n$, $X\backslash
A_{n+1}$ and $B_n$, $X\backslash B_{n+1}$ are completely separated
in $X$. Then the following conditions  are equivalent.
\begin{itemize}
\item[\rm(1)] ${\mathcal T}^*(X)$ and ${\mathcal T}^*(Y)$ are order-isomorphic;
\item[\rm(2)] $\beta X\backslash\upsilon X$ and $\beta Y\backslash\upsilon Y$ are homeomorphic.
\end{itemize}
\end{theorem}

\begin{proof}
By Theorem \ref{YOOG} $\upsilon X$ is locally compact and
$\sigma$-compact, and therefore by 1B of [13], we have $\beta
X\backslash\upsilon X\in {\mathcal Z}(\beta X)$. Let $Z\in {\mathcal
Z}(\beta X\backslash\upsilon X)$. Then since $\upsilon X$ is locally
compact, $\beta X\backslash\upsilon X$ is closed in $\beta X$, and
therefore there exists an $S\in {\mathcal Z}(\beta X)$ such that
$Z=S\cap(\beta X\backslash\upsilon X)$. Thus $Z\in{\mathcal Z}(\beta
X)$. Clearly $Z\cap X=\emptyset$, and therefore by Theorem \ref{AADJ} we
have ${\mathcal Z}(\beta X\backslash\upsilon X)\subseteq
\mu_X({\mathcal T}^*(X))\cup\{\emptyset\}$. Clearly for every
$C\in\mu_X({\mathcal T}^*(X))$, since $C\in{\mathcal Z}(\beta X)$
and $C\cap X=\emptyset$, we have $C\subseteq \beta
X\backslash\upsilon X$. Therefore ${\mathcal Z}(\beta
X\backslash\upsilon X)= \mu_X({\mathcal T}^*(X))\cup\{\emptyset\}$.
Similarly ${\mathcal Z}(\beta Y\backslash\upsilon Y)=
\mu_Y({\mathcal T}^*(Y))\cup\{\emptyset\}$. Now since $\mu_X$ and
$\mu_Y$ are order-anti-isomorphisms, ${\mathcal T}^*(X)$ and
${\mathcal T}^*(Y)$ are order-isomorphic, if and only if, ${\mathcal
Z}(\beta X\backslash\upsilon X)$ and ${\mathcal Z}(\beta
Y\backslash\upsilon Y)$ are order-isomorphic, if and only if, $\beta
X\backslash\upsilon X$ and $\beta Y\backslash\upsilon Y$ are
homeomorphic.
\end{proof}

\section{Some cardinality theorems}

Suppose that $X$ is a locally compact space. Let $w(T)$ and $d(T)$
denote the  weight and the density of a space $T$, respectively.
Then since
\[w(X^*)\leq w(\beta X)\leq 2^{ d(\beta X)}\leq 2^{ d(X)}\]
we have
\[\big|{\mathcal T}(X)\big|\leq\big|{\mathcal C}(X^*)\big|\leq 2^{w(X^*)}\leq
2^{2^{d(X)}}\]
which gives an upper bound for cardinality of the set
${\mathcal T}(X)$. In the following theorems we obtain a lower bound
for cardinalities of two subsets of ${\mathcal T}(X)$. Here for a
space $T$,  $L(T)$ denotes the Lindel\"{o}f number of $T$.

\begin{theorem}\label{DFG}
Let $X$ be a locally compact
paracompact non-compact space. Then
\[2^{L(X)}\leq \big|{\mathcal T}_L(X)\big|.\]
\end{theorem}

\begin{proof}
{\em Case 1)} Suppose that $X$ is $\sigma$-compact.
Then since $X$ is non-pseudocompact, as $X$ paracompact and
non-compact (see Theorem \ref{YUGG}.20 of [4]) by 4C of [13] we have
$|X^*|\geq 2^{2^{\aleph_0}}$. Now since each element of ${\mathcal
T}(X)$ is $\sigma$-compact, ${\mathcal T}_L(X)={\mathcal T}(X)$, and
thus we have
\[ \big|{\mathcal
T}_L(X)\big|=\big|{\mathcal T}(X)\big|\geq\big|\big\{X\cup\{p\}:p\in X^*\big\}\big|=|X^*|\geq
2^{2^{\aleph_0}}\geq 2^{L(X)}.\]

{\em Case 2)} Suppose that $X$ is non-$\sigma$-compact. Assume the
notations of Proposition \ref{RAWDJ}. Then since each $X_i$ is
$\sigma$-compact, we have $L(X)\leq |I|$. For each $J\subseteq I$,
let  $Q_J=\bigcup_{i\in J}X_i$ and $C_J=Q_J^*\cup(\beta
X\backslash\sigma X )$. For $J_1, J_2\subseteq I$, if $j\in
J_1\backslash J_2$, then since $X^*_j\subseteq C_{J_1}$ and
$X^*_j\cap C_{J_2}=\emptyset$, we have $C_{J_1}\neq C_{J_2}$. By
Lemma \ref{ERWRPJ}, for each $J\subseteq I$, there exists $Y_J\in {\mathcal
T}_L(X)$ such that $\mu (Y_J)=C_J$.  Now
\[\big|{\mathcal T}_L(X)\big|\geq\big| {\mathcal P }(I)\big|=2^{|I|}\geq 2^{L(X)}.\]
\end{proof}

For purpose of the next result we need the following proposition
stated in Lemma 3\ref{TYUGG} of [9].

\begin{proposition}\label{EEFR}
Suppose that $E$ is an infinite set of cardinality $\alpha$. Then there exists a collection $\mathcal{A}$ of subsets of $E$ with $|{\mathcal A}|=2^{\alpha}$ such that for any distinct $A, B\in \mathcal{A}$ we have $|A\backslash B|=\alpha$.
\end{proposition}

\begin{theorem}\label{DFGG}
Let $X$ be a locally compact
paracompact non-$\sigma$-compact space. Then
\[2^{L(X)}\leq \big|{\mathcal T}_D(X)\big|.\]
\end{theorem}

\begin{proof}
Assume the notations of Proposition \ref{RAWDJ}. By the above
proposition, since $\alpha=|I|>\aleph_0$, there exists a family
$\{J_s\}_{s\in S}$ of subsets of $I$, faithfully indexed, such that
$|S|=2^{\alpha}$ and $|J_s\backslash J_t|=\alpha$, for distinct
$s,t\in S$. For each $s\in S$, let $Q_s=\bigcup_{i\in J_s}X_i$ and
let $C_s= Q_s^*\backslash\sigma X$. If for some $s\in S$, we have
$C_s= \emptyset$, then since $\mbox{cl}_{\beta X} Q_s\subseteq\sigma
X$, we have $\mbox{cl}_{\beta X} Q_s\subseteq\mbox{cl}_{\beta
X}(\bigcup_{i\in H}X_i)$, for some countable $H\subseteq I$, and
thus $Q_s\subseteq\bigcup_{i\in H}X_i$, as  $ \bigcup_{i\in H}X_i$
is clopen in $X$. But this is a contradiction as $J_s$ is not
countable. Therefore $C_s\neq\emptyset$ for any $s\in S$. By Theorem
\ref{EPSDJ} for each $s\in  S$, we have $C_s=\mu (Y_s)$ for some
$Y_s\in{\mathcal T}_D(X)$. Suppose that $s,t\in S$ and $s\neq t$.
Let $K=J_s\backslash J_t$ and let $P=\bigcup_{i\in K}X_i$. Then
since $|K|=\alpha$, we have $A=P^*\backslash \sigma X\neq\emptyset$.
But since $P\cap Q_t=\emptyset$, we have $P^*\cap C_t=\emptyset$,
which implies that $A\cap C_t=\emptyset$. Therefore since
$\emptyset\neq A\subseteq C_s$, we have  $C_s\neq C_t$. Thus for any
distinct $s,t\in S$, we have $Y_s\neq Y_t$. This shows  that
$|{\mathcal T}_D(X)|\geq |S|=2^\alpha$, which together with the fact
that $\alpha=|I|\geq L(X)$ proves the
theorem.
\end{proof}

\section{Some applications}

In this section we correspond to each one-point extension of a
Tychonoff space $X$ an ideal of $C^*(X)$. Using this, and applying
some of our previous results, we will be able to obtain some
relations between the order structure of certain collections of
ideals of $C^*(X)$, partially ordered by set-theoretic inclusion,
and the topology of a certain subspace of $X^*$.

For a Tychonoff space
$X$, let ${\mathcal I}(C^*(X))$ denote the set of all ideals of
$C^*(X)$.  We define a function
\[\gamma:\big({\mathcal T}(X),\leq\big)\rightarrow\big({\mathcal
I}\big(C^*(X)\big),\subseteq\big)\]
by
\[ \gamma (Y)=\big\{f|X:f\in C^*(Y) \mbox { and } f(p)=0\big\}\]
for $Y=X\cup\{p\}\in{\mathcal T}(X)$.

\begin{lemma}\label{YSDG}
The function $\gamma$ is an
order-isomorphism onto its image.
\end{lemma}

\begin{proof}
To show that $\gamma$ is well-defined, consider the
functions  $g\in \gamma (Y) $ and $h\in C^*(X)$. Let
$f:Y\rightarrow {\bf R}$ be defined such that $f(p)=0$ and
$f|X=g.h$. We verify that $f$ is continuous. So let $G\in C^*(Y)$ be
such that $G(p)=0$ and $G|X=g$. Suppose that $\epsilon >0$. Let $W$
be an open neighborhood of $p$ in $Y$ such that $G(W)\subseteq
(-\epsilon/M,\epsilon/M)$, where $M>0$ and $|h(x)|\leq M$ for every
$ x\in X$. Then for every  $x\in W\cap X$ we have $|f(x)|
=|g(x)|<\epsilon$. So $f$ is continuous. Now $g.h=f| X\in \gamma
(Y)$. It is clear that for any $k,l\in \gamma(Y)$, $k-l\in
\gamma(Y)$. This shows that $\gamma$ is well-defied.

Now suppose that $Y_i=X\cup\{p_i\}\in{\mathcal T}(X)$, for
$i=1,2$. Suppose that $Y_1\geq Y_2$ and let $\phi: Y_1\rightarrow
Y_2$ be a continuous function such that $\phi|X=\mbox {id}_X$. Let $
g\in\gamma (Y_2) $. Then $g=f|X$, where $f\in C^*(Y_2)$ and
$f(p_2)=0$. Now since $\phi(p_1)=p_2$, we have $g=f|X=f\phi|X\in
\gamma (Y_1) $, i.e., $\gamma (Y_1)\supseteq \gamma (Y_2)$.

Conversely, suppose that $\gamma (Y_1)\supseteq \gamma (Y_2)$.
Define a function  $\phi : Y_1\rightarrow Y_2$ by $\phi|X=\mbox
{id}_X$ and $\phi(p_1)=p_2$. To show that $\phi$ is continuous at
$p_1$, suppose that $V$ is an open neighborhood of $p_2=\phi(p_1)$
in $Y_2$. Let $f:Y_2\rightarrow {\bf I}$ be a continuous function
such that $f(p_2)=0$ and $ f(Y_2\backslash V)\subseteq \{1\}$. Then
since $f|X\in\gamma (Y_2)$, we have $f|X\in\gamma (Y_1)$ and
therefore $f|X=h|X$, for some $h\in C^*(Y_1)$ with $h(p_1)=0$. Now
$U=h^{-1}([0,1))$ is an open neighborhood of $p_1$ in $Y_1$
satisfying $\phi(U)\subseteq V$. This proves the continuity of
$\phi$ and therefore we have  $ Y_1\geq
Y_2$.
\end{proof}

The following result is well known. We include a proof in here for
the sake of completeness.

\begin{lemma}\label{ASDG}
Let $X$ be a strongly
zero-dimensional locally compact space. Then the set of  clopen
subset of $X^*$ consist of exactly those sets which are of the form
$U^*$, for some clopen subset $U$ of $X$.
\end{lemma}

\begin{proof}
Clearly for every clopen subset $U$ of $X$, the set
$U^*$ is clopen in $X^*$. To see the converse suppose that $C$ is a
clopen subset of $X^*$. Let $W$ be an open set of $\beta X$ such
that $C=W\backslash X$. Since $C\subseteq X$ is compact, there
exists a clopen subset $V$ of $\beta X$ such that $C\subseteq
V\subseteq W$, and therefore
\[C=\mbox {cl}_{\beta X} V\backslash X=\mbox {cl}_{\beta X} (V\cap
X)\backslash X=(V\cap X)^*.\]
\end{proof}

For a Tychonoff space $X$ and $E\subseteq X$, we let
\[ I_E=\big\{g\in C^*(X): |g|^{-1}\big([\epsilon,\infty)\big)\backslash E\mbox{ is compact for any } \epsilon>0\big\}.\]
It is easy to see that if $E$ is open in $X$ then $I_E$ is
an ideal in $C^*(X)$.

\begin{lemma}\label{FFDG}
 Let $X$ be a locally compact space
and let $U$ be a clopen subset of $X$. If $Y\in{\mathcal T}(X)$ is
such that $\mu (Y)=X^*\backslash U^*$ then $\gamma (Y)=I_U$.
\end{lemma}

\begin{proof}
Suppose that $g\in \gamma(Y)$. Then $g=f|X$ for some
$f\in C^*(Y)$ with $f(p)=0$. Suppose that there exists an
$\epsilon>0$ such that $G=|g|^{-1}([\epsilon, \infty))\backslash U$
is not compact, and let $x\in G^*$. By continuity of $f$ there
exists an open neighborhood $W$ of $p$ in $\beta Y$ such that
$f(W\cap Y)\subseteq (-\epsilon, \epsilon)$. Since $p\in W$,
$X^*\backslash U^*=q^{-1}(p)\subseteq q^{-1}(W)$, where $q:\beta
X\rightarrow \beta Y$ is the quotient map contracting $X^*\backslash
U^*$ to the point  $p$. Now since $x\in \mbox {cl}_{\beta
X}G\subseteq \mbox {cl}_{\beta X}(X\backslash U) $ and $U$ is clopen
in $X$, where $x\notin \mbox {cl}_{\beta X}U$ and thus $x\in
X^*\backslash U^*\subseteq q^{-1}(W)$, which implies that $G\cap
q^{-1}(W)\neq\emptyset$.  Let $t\in G\cap q^{-1}(W)$. Then since
$t\in G$, we have $|f(t)| =|g(t)|\geq\epsilon$. But on the other
hand, $t=q(t)\in W$ and by the way we chose $W$, we have
$|f(t)|<\epsilon$. This is a contradiction, therefore
$|g|^{-1}([\epsilon, \infty))\backslash U$ is compact for any $\epsilon>0$, and thus $g\in I_U$. This shows that
$\gamma(Y)\subseteq I_U$.

Conversely, let $g\in I_U$ and
define a function $f:Y\rightarrow {\bf R}$ such that $f| X=g$ and
$f(p)=0$. We verify that $f$ is continuous.  So suppose that
$\epsilon>0$. Since $U$ is clopen in $X$, the set  $X^*\backslash
U^*$ is compact and it is disjoint from the compact subset
$G=|g|^{-1}([\epsilon, \infty))\backslash U$ of $X$. Let $W$ be an
open neighborhood of $X^*\backslash U^*$ in $\beta X$ disjoint from
$G$. Then $V=W\backslash \mbox {cl}_{\beta X} U$ is an open set of
$\beta X$ containing $X^*\backslash U^*$. Clearly $T=(V\backslash
(X^*\backslash U^*))\cup\{p\}$ is open in $\beta Y$. Now $T\cap Y$
is an open neighborhood of $p$ in $Y$ such that $f(T\cap Y)\subseteq
(-\epsilon, \epsilon)$. This is because if $t\in T\cap X$, then
since $T\cap X\subseteq V$ and $W\cap G=\emptyset$, we have $t\notin
G$ and $t\notin \mbox {cl}_{\beta X} U$, and therefore $t\notin
|g|^{-1}([\epsilon, \infty))$. Thus $|f(t)|=|g(t)|<\epsilon$. This
shows that $f$ is continuous and therefore $g=f| X\in \gamma(Y)$,
i.e., $I_U\subseteq \gamma (Y)$.
\end{proof}

For a Tychonoff space $X$, let
\[\Sigma_X=\{I_U: U \mbox { is a $\sigma$-compact clopen  subset of }
X\}.\]

\begin{lemma}\label{FGTTG}
Let $X$ be  a zero-dimensional
locally compact paracompact non-$\sigma$-compact space. Then
\[\gamma\big({\mathcal T}_{KL}(X)\big)=\Sigma_X.\]
\end{lemma}

\begin{proof}
Suppose that  $Y\in {\mathcal T}_{KL}(X)$  and let
$C=\mu(Y)$. Then by Theorem \ref{TODJ} and Lemma \ref{ERWRPJ} we have  $C$ is
clopen in $X^*$ and contains $\beta X\backslash \sigma X$. Now $X$,
being zero-dimensional, locally compact and paracompact, is strongly
zero-dimensional (see Theorem  \ref{EEFR}.10 of [4]) and thus by Lemma \ref{ASDG}
we have  $X^*\backslash C=U^*$, for some clopen $U\subseteq X$. By
Lemma \ref{FFDG} we have$\gamma(Y)=I_U$. But since $X^*\backslash
U^*=C\supseteq\beta X\backslash\sigma X$, we have $\mbox {cl}_{\beta
X}U\subseteq\sigma X$ and thus $U$ is $\sigma$-compact. This shows
that $\gamma(Y)\in \Sigma_X$, i.e., $\gamma({\mathcal
T}_{KL}(X))\subseteq\Sigma_X$.

Conversely, if $U$ is a $\sigma$-compact clopen subset of $X$, then
$\mbox {cl}_{\beta X}U$ is clopen in $\beta X$ and is contained in
$\sigma X$, therefore $C=X^*\backslash U^*\supseteq \beta
X\backslash\sigma X$. If we let $\mu (Y) =C$, then $Y\in{\mathcal
T}_{KL}(X)$, and by Lemma \ref{FFDG} we have $\gamma(Y)=I_U$, which shows
that $\Sigma_X\subseteq \gamma({\mathcal
T}_{KL}(X))$.
\end{proof}

For a Tychonoff space $X$, let
\[\Delta_X=\{I_{X\backslash U}: U \mbox{ is a $\sigma$-compact non-compact  clopen  subset of }X\}.\]

\begin{lemma}\label{OOG}
Let $X$ be a zero-dimensional
locally compact paracompact non-$\sigma$-compact space. Then
\[\gamma\big({\mathcal T}_{K}^*(X)\big)=\Delta_X.\]
\end{lemma}

\begin{proof}
Suppose that  $Y\in {\mathcal T}_{K}^*(X)$. By Lemma \ref{RAJ} the set  $\mu(Y)$ is clopen in $X^*$. Now $X$ being
zero-dimensional, locally compact and paracompact,  is strongly
zero-dimensional (see Theorem \ref{EEFR}.10 of [4]) therefore by Lemma \ref{ASDG}
we have  $\mu(Y)=U^*$, for some clopen subset $U$ of $X$. By Lemma \ref{RWRPJ} we have $U^*=\mu (Y)\subseteq\sigma X$, which implies that $U$
is $\sigma$-compact and thus by Lemma \ref{FFDG}, we have
$\gamma(Y)=I_{X\backslash U}\in \Delta_X$.

For the converse, let $U$
be a $\sigma$-compact non-compact clopen subset of $X$. Then by
Lemma 16 we have  $U^*=\mu (Y)$, for some $Y\in{\mathcal
T}_{K}^*(X)$. Now since $\mu (Y)=X^*\backslash(X\backslash U)^*$, by
Lemma \ref{FFDG}, we have  $I_{X\backslash U}=\mu (Y)\in \gamma({\mathcal
T}_{K}^*(X))$.
\end{proof}

\begin{theorem}\label{FGG}
For zero-dimensional locally
compact paracompact non-$\sigma$-compact spaces $X$ and $Y$ the
following conditions are equivalent.
\begin{itemize}
\item[\rm(1)] $(\Sigma_X, \subseteq)$ and $(\Sigma_Y, \subseteq)$ are order-isomorphic;
\item[\rm(2)] $(\Delta_X, \subseteq)$ and $(\Delta_Y, \subseteq)$ are order-isomorphic;
\item[\rm(3)] $\sigma X\backslash X$ and $\sigma Y\backslash Y$ are homeomorphic.
\end{itemize}
\end{theorem}

\begin{proof}
This follows from Lemmas \ref{YSDG}, \ref{FGTTG} and \ref{OOG}, Corollary
\ref{UPJ} and Theorem \ref{PSDJ}.
\end{proof}

\begin{definition}\label{FFGA}
Let $X$ be a Tychonoff space. A
sequence $\{U_n\}_{n<\omega}$ is called a {\em  $\sigma$-regular
sequence of open sets} in $X$, if for each $n<\omega$, the set $U_n$
is open in $X$ and is  such that $\mbox {cl}_XU_n $ is
$\sigma$-compact and non-compact, and $U_n \supseteq\mbox
{cl}_XU_{n+1} $.

If $ {\mathcal U}=\{U_n\}_{n<\omega}$ is  a $\sigma$-regular
sequence of open sets in $X$, we let $I_{\mathcal U}$ denote the set
\[\big\{ g\in C^*(X):\mbox {for any $\epsilon>0$, $|g|^{-1}\big([\epsilon, \infty)\big)\cap\mbox
{cl}_XU_n $ is compact for some $n<\omega$}\big\}\]
and let
\[\Omega_X=\{I_{\mathcal U}:\mbox{${\mathcal U}$ is  a
$\sigma$-regular sequence of open sets in $X$} \}.\]
\end{definition}

\begin{lemma}\label{LLG}
Let $X$ be a  locally compact
paracompact non-$\sigma$-compact space and let  $ {\mathcal
U}=\{U_n\}_{n<\omega}$ be  a $\sigma$-regular sequence of open sets
in $X$. Suppose that  $ \{f_n\}_{n<\omega}$ is  a sequence in
$C(\beta X, \mathbf{I})$ such that for each $n<\omega$, we have
$f_n(U_{n+1})\subseteq\{0\}$ and $f_n(X\backslash
U_{n})\subseteq\{1\}$. Then $C=\bigcap_{n<\omega}Z(f_n)\backslash
X\in \mu ({ \mathcal T}(X))$, and if $Y\in { \mathcal T}(X)$ is such
that $\mu (Y)=C$, then we have $\gamma(Y)=I_{\mathcal U}$.
\end{lemma}

\begin{proof}
First note that since for each $n<\omega$, we have
$\mbox{cl}_{\beta X}U_{n+1}\subseteq Z(f_n)$, and
$\emptyset\neq\bigcap_{n<\omega}U_n^*\subseteq C$ and so $C\in\mu ({
\mathcal T}(X))$.  Suppose that $C=\mu (Y)$, for some
$Y=X\cup\{p\}\in { \mathcal T}(X)$. Let  $g\in\gamma(Y)$. Then $g=f|
X$ for some  $f\in C^*(Y)$  with $f(p)=0$. Suppose that $g\notin
I_{\mathcal U}$. Then there exists an $\epsilon>0$ such that for
each $n<\omega$, the set  $A_n=|g|^{-1}([\epsilon,
\infty))\cap\mbox {cl}_XU_n $ is not compact. By compactness of
$X^*$ we have $\bigcap_{n<\omega}A_n^*\neq\emptyset$. Let $x\in
\bigcap_{n<\omega}A_n^*$. Then $x\in C$. Let  $Z$ be the space
obtained from $\beta X$ by contracting $C$ to the point $p$, and let
$q:\beta X\rightarrow Z=\beta Y$ be its natural quotient mapping. By
continuity of $f$ there exists an open neighborhood $V$ of $p$ in
$Y$ such that $f(V)\subseteq (-\epsilon, \epsilon)$. Let $W$ be an
open subset of $\beta Y$ with $W\cap Y=V$. Then since $p\in W$, the
set  $C\subseteq q^{-1}(W)$. Now since
$\bigcap_{n<\omega}A_n^*\subseteq C$, the set $ q^{-1}(W)$ is an
open neighborhood of $x$ in $\beta X$, and therefore since $x\in
\mbox{cl}_{\beta X}A_1$, we have $A_1\cap q^{-1}(W)\neq\emptyset $.
Let $t\in A_1\cap q^{-1}(W)$. Then $t=q(t)\in W$ and thus
$|g(t)|=|f(t)|<\epsilon$. But since $t\in A_1$, we have $|g(t)|\geq
\epsilon$, which is a contradiction. This shows that $g\in
I_{\mathcal U}$. Thus $\gamma(Y)\subseteq I_{\mathcal U}$.

To show the reverse inclusion, let $g\in I_{\mathcal U}$. Define a
function  $f:Y\rightarrow\mathbf{R}$ such that $f| X=g$ and
$f(p)=0$. We show that $f$ is continuous at $p$. So let $\epsilon
>0$. Then by assumptions  there exists a $k<\omega$ such that
$S=|g|^{-1}([\epsilon, \infty))\cap\mbox {cl}_XU_k $ is compact.
Since $C\subseteq f_k^{-1}([0, 1))$, the set
$T=(f_k^{-1}([0,1))\backslash C)\cup\{p\}$ is open in $\beta Y$.
Consider the open neighborhood $(T\cap Y)\backslash S$ of $p$ in
$Y$. If $t\in (T\cap Y)\backslash S$, then since $t\in T$, we have
$f_k(t)
<1$ and so $t\in U_k$. But $t\notin S$ and therefore
$|f(t)|=|g(t)|<\epsilon$,
i.e., $f((T\cap Y)\backslash S)\subseteq(-\epsilon,\epsilon)$. This
shows the continuity of $f$ and thus $g\in \gamma (Y)$. Therefore
$I_{\mathcal U}\subseteq\gamma(Y)$, which together with the previous
part of the proof proves the lemma.
\end{proof}

\begin{lemma}\label{GG}
Let $X$ be a  locally compact
paracompact non-$\sigma$-compact space. Then
\[\gamma\big({\mathcal T}^*_S(X)\big)= \Omega_X.\]
\end{lemma}

\begin{proof}
Assume the notations of Proposition \ref{RAWDJ}. Suppose that
$Y=X\cup\{p\}\in{\mathcal T}^*_S(X)$. By Lemma \ref{UPPJ}, we have
$\mu(Y)=C\in {\mathcal Z}(\beta X)$ and $C\subseteq \sigma X$.
Therefore $C\subseteq \mbox {cl}_{\beta X} M$, where
$M=\bigcup_{i\in J} X_i$, for some countable $J\subseteq I$. Let
$h\in C(\beta X, \mathbf{I})$ be such that $Z(h)=C$ and $h(\beta
X\backslash \mbox {cl}_{\beta X} M)\subseteq\{1\}$. For each
$n<\omega$ let $U_n =h^{-1}([0,1/n))\cap X$. Then since for each
$n<\omega$, $C\subseteq h^{-1}([0,1/n))$, we have
$U_n\neq\emptyset$, and since $U_n\subseteq M$ and $C\subseteq \mbox
{cl}_{\beta X}U_n$, $\mbox {cl}_XU_n$ is $\sigma$-compact and
non-compact. Clearly for each  $n<\omega$, we have $U_n
\supseteq\mbox {cl}_XU_{n+1} $, which shows that $ {\mathcal
U}=\{U_n\}_{n<\omega}$ is  a $\sigma$-regular sequence of open sets
in $X$.

For each $n<\omega$ define $f_n:\beta X\rightarrow \mathbf{I}$ by
\[ f_n=\Big(\Big(\Big(h\wedge
\frac{1}{n}\Big)\vee\frac{1}{n+1}\Big)-\frac{1}{n+1}\Big)\Big(\frac{1}{n}-\frac{1}{n+1}\Big)^{-1}.\]
Then the  sequence $ \{f_n\}_{n<\omega}$ satisfies the
requirements of Lemma \ref{LLG}, and therefore since
$\mu(Y)=Z(h)=\bigcap_{n<\omega} Z(f_n)$, we have $ \gamma
(Y)=I_{\mathcal U}$. This shows that  $\gamma ({\mathcal
T}^*_S(X))\subseteq \Omega_X$.

To complete the proof we need to show that $\Omega_X\subseteq\gamma
({\mathcal T}^*_S(X))$. So let $ {\mathcal U}=\{U_n\}_{n<\omega}$ be
a $\sigma$-regular sequence of open sets in $X$. We verify that
$I_{\mathcal U}\in \gamma ({\mathcal T}^*_S(X))$. For each
$n<\omega$ since $U_n \supseteq\mbox {cl}_XU_{n+1} $, by normality
of $X$, there exists an $f_n\in C(X, \mathbf{I})$ such that $f_n(
\mbox{cl}_XU_{n+1})\subseteq\{0\}$ and $f_n(X\backslash
U_n)\subseteq \{1\}$. Let $F_n\in C(\beta X, \mathbf{I})$ be the
extension of $f_n$. Since each $\mbox{cl}_XU_{n+1}$ is
$\sigma$-compact, for each $n<\omega$ we have $U_n\subseteq P$,
where $P=\bigcup_{i\in L} X_i$ and $L\subseteq I$ is countable.
Since $ X\backslash P\subseteq X\backslash U_n\subseteq
F_n^{-1}(1)$, it follows that $Z(F_n)\subseteq\beta
X\backslash\mbox{cl}_{\beta X} (X\backslash P)=\mbox{cl}_{\beta X}
P$ and therefore since $P^*\in {\mathcal Z} (\beta X)$ we have
$D=\bigcap_{n<\omega}Z(F_n)\backslash X\in {\mathcal Z} (\beta X)$.
But $\emptyset\neq U_n^*\subseteq Z(F_n)$ and $D\subseteq \sigma X$,
which by Lemma \ref{UPPJ} implies that $D=\mu (Y)$, for some
$Y=X\cup\{p\}\in{\mathcal T}^*_S(X)$. Now the sequence $
\{F_n\}_{n<\omega}$  satisfies the requirements of Lemma \ref{LLG} and
therefore $\gamma (Y)=I_{\mathcal U}$. This shows that
$\Omega_X\subseteq\gamma ({\mathcal T}^*_S(X))$, which
together with the first part of the proof  give the
result.
\end{proof}

Now from Theorem \ref{ROEJ} and the above lemmas we obtain the following
result.

\begin{theorem}\label{IUUG}
For  locally compact paracompact
non-$\sigma$-compact spaces $X$ and $Y$ the following conditions are
equivalent.
\begin{itemize}
\item[\rm(1)] $(\Omega_X, \subseteq)$ and $(\Omega_Y, \subseteq)$ are order-isomorphic;
\item[\rm(2)] $\sigma X\backslash X$ and $\sigma Y\backslash Y$ are homeomorphic.
\end{itemize}
\end{theorem}

\noindent {\bf  Acknowledgments.} The author would like to express
his deep gratitude to Professor R. Grant Woods for his invaluable
comments during this work. The author also thanks the referee for
reading the manuscript and his comments.

\end{document}